\documentclass{amsart}

\usepackage[arrow, matrix, curve]{xy}
\usepackage{verbatim}

\newtheorem{thm}{Theorem}[section]
\newtheorem{lem}[thm]{Lemma}
\newtheorem{cor}[thm]{Corollary}

\theoremstyle{definition}
\newtheorem{dfn}[thm]{Definition}

\newtheorem{rem}[thm]{Remark}
\newtheorem{nte}[thm]{Note}

\theoremstyle{remark}

\newtheorem*{exl}{Example}
\newtheorem*{red}{Reminder}

\newcommand{\N}{\mathbb{N}}
\newcommand{\R}{\mathbb{R}}
\newcommand{\Z}{\mathbb{Z}}
\newcommand{\F}{\mathbb{F}}
\newcommand{\To}{\rightarrow}

\newcommand{\MTo}{\mapsto}

\newcommand{\mr}{\mathrm}
\newcommand{\f}{\mathcal{F}}

\newcommand{\mf}{\mathfrak}

\begin{document}

\title{Taut Submanifolds and Foliations}
\author{Stephan Wiesendorf}

\address{Stephan Wiesendorf\\
Mathematisches Institut, Universit\"at zu K\"oln, Weyertal 86-90, 50931 K\"oln, Germany}
\email{swiesend@math.uni-koeln.de}

\footnote{2010 \emph{Mathematics Subject Classification.} Primary 53C42, Secondary 53C12. }

\begin{abstract}
We give an equivalent description of taut submanifolds of complete Riemannian manifolds as exactly those submanifolds whose normal exponential map has the property that every preimage of a point is a union of submanifolds. It turns out that every taut submanifold is also $\mathbb Z_2$-taut. We explicitely construct generalized Bott-Samelson cycles for the critical points of the energy functionals on the path spaces of a taut submanifold which, generically, represent a basis for the $\mathbb Z_2$-cohomology. We also consider singular Riemannian foliations all of whose leaves are taut. Using our characterization of taut submanifolds, we are able to show that tautness of a singular Riemannian foliation is actually a property of the quotient.
\end{abstract}

\maketitle

\tableofcontents

\section{Introduction}
\label{sec:Intro}
The terminology of taut submanifolds was introduced by Carter and West in \cite{CW}, where they call a submanifold $L$ of a Euclidean space $V$ \emph{taut} if all the squared distance functions $d^2_q: L \rightarrow \mathbb R$, $d^2_q(p) = \|p -q\|^2$, with respect to points $q \in V$ that are not focal points of $L$ are perfect Morse functions for some field $\mathbb F$, i.e. if the number of critical points of index $k$ of $d^2_q$ coincides with the $k$-th Betti number of $L$ with respect to the field $\mathbb F$ for all $k$. If $L$ is taut and $\mathbb F$ is a field as in the defintion of tautness, then $L$ is also called $\mathbb F$-\emph{taut}.
Thus, geometrically, taut submanifolds are as round as possible. If one tries to generalize this defintion to submanifolds of arbitrary Riemannian manifolds, the problem arises that the squared distance function is not a priori everywhere smooth anymore. Namely, it is not differentiable in the intersection of the cut locus of the respective point $q$ with the submanifold.
 
Using different approaches, Grove and Halperin \cite{GH} and, independently, Terng and Thorbergsson \cite{TT} generalized this notion to submanifolds $L$ of complete Riemannian manifolds $M$ by saying that $L$ is taut if there exists a field $\mathbb F$ such that every energy functional $E_q(c) = \int_{[0,1]}\|\dot c(t)\|^2dt$ on the space $\mathcal P(M,L \times q)$ of $H^1$-paths $c: [0,1] \rightarrow M$ from $L$ to a fixed point $q \in M$ is a perfect Morse function with respect to $\mathbb F$ if $q$ is not a focal point of $L$. The critical points of $E_q$ are exactly the geodesics that start orthogonally to $L$ and end in $q$. In particular, in the case where $M = V$ is a Euclidean space, there is an obvious way to identify a submanifold $L$ with the space of segments in $\mathcal P(V,L \times q)$ and under this identification the map $d_q^2$ corresponds to $E_q$. Further, it is not hard to see that in this case the path space $\mathcal P(V,L \times q)$ admits the subspace of segments from $L$ to $q$ as a strong deformation retract. So the definitions agree for submanifolds of a Euclidean space and it turns out that this is indeed the right way to generalize tautness.

It is shown in \cite{TT} that if $L \hookrightarrow M$ is an $\mathbb F$-taut submanifold, then the energy functionals $E_q$ are Morse-Bott functions for all points $q \in M$. Our first main result now states that this property actually characterize taut submanifolds.\\

\noindent \textbf{Theorem A} 
\emph{A closed submanifold $L$ of a complete Riemannian manifold $M$ is taut if and only if all the energy functionals $E_q: \mathcal P(M,L \times q) \rightarrow \mathbb R$ are Morse-Bott functions.}\\

In fact, if all the energy functionals are Morse-Bott functions, then the field with respect to which $L$ is taut is $\mathbb Z_2$. Thus, as a direct consequence, we obtain the following result, which was, just as Theorem A, so far not even known in the case of a Euclidean space.\\

\noindent \textbf{Theorem B}
\emph{If a submanifold is $\mathbb F$-taut, then it is also $\mathbb Z_2$-taut.}\\

Based on this result it is reasonable to consider only $\mathbb Z_2$-taut submanifolds, so that we no longer distinguish between $\mathbb Z_2$-taut and taut.\\  

As the definition shows, tautness is a very special property. In some sense, it is a kind of homogeneity requirement for the pair $(M,L)$. So it is no surprise that so far not many examples of taut submanifolds are known. This makes it all the more remarkable that taut submanifolds, if at all, often occur in families which then decompose the ambient space, e.g. an orbit decomposition induced by the isotropy representation of a symmetric space. It is for this reason that we study such families as they usually appear, i.e. singular Riemannian foliations with only taut leaves. We call such families \emph{taut foliations}. As a main result in this direction we observe that tautness of a foliation is indeed a property of the quotient of the foliation, so that it actually makes sense to talk about taut quotients as equivalence classes of quotients under isometries.\\

\noindent \textbf{Theorem C}
\emph{Let $\mathcal F$ and $\mathcal F'$ be closed singular Riemannian foliations on complete Riemannian manifolds $M$ and $M'$ such that their quotients $M/\mathcal F$ and $M'/\mathcal F'$ are isometric. Then $\mathcal F$ is taut if and only if $\mathcal F'$ is taut.}\\

Due to this result one could think about taut foliations as foliations with \emph{pointwise taut} quotients, where we follow \cite{Lei} and call a manifold pointwise taut if all of its points are taut (submanifolds). Of course, in general a quotient of a singular Riemannian foliation is far from being a manifold, but as soon as it is a nice space in the sense that one could use differential geometric methods it turns out that this picture is reasonable. Viewed in this light, the largest class of spaces for which one has the appropriate tools available is the class of Riemannian orbifolds, i.e. spaces locally modeled by quotients of Riemannian manifolds modulo the action of a finite group of isometries. Now, given a taut singular Riemannian foliation $\mathcal F$ on $M$ such that the quotient $M/\mathcal F$ is an orbifold, it follows that $M/\mathcal F$ is already a good Riemannian orbifold, that is to say $M/\mathcal F$ is isometric to $N/\Gamma$, where $N$ is a Riemannian manifold and $\Gamma \subset \mr{Iso}(N)$ is a discrete group of isometries. This observation together with the last theorem leads directly to our next result, which mainly motivates our picture of pointwise taut quotients.\\

\noindent \textbf{Theorem D}
\emph{Let $\mathcal F$ be a closed singular Riemannian foliation on a complete Riemannian manifold $M$. Then $\mathcal F$ is taut and $M/\mathcal F$ is an orbifold if and only if $M/\mathcal F$ is isometric to $N/\Gamma$, where $N$ is a pointwise taut Riemannian manifold and $\Gamma \subset \mr{Iso}(N)$ is a discrete group of isometries of $N$.}\\

In view of applications, the more interesting direction of this result is that the existence of a pointwise taut quotient covering implies tautness of the foliation. The known examples of pointwise taut Riemannian manifolds are mainly two classes of spaces, together with Riemannian products of elements of these classes and their subcoverings. The first one is the class of symmetric spaces, which are pointwise taut by the work of Bott and Samelson \cite{BS} and the second class consists of manifolds without conjugate points, e.g. manifolds with non-positive curvature. In fact, if there are no conjugate points along any geodesic in a Riemannian manifold, the index of every critical point of a given energy functional is zero, hence all points in such a manifold are taut. A conjecture in \cite{TT} states that a compact pointwise taut Riemannian manifold that has the homotopy type of a compact symmetric space is symmetric. We want to mention as an aside that it is shown in \cite{TT} that in the case of a compact rank-one symmetric space this conjecture is equivalent to the Blaschke conjecture which is still not settled. 

It is therefore not a surprise that in all known examples of taut foliations with orbifold quotients these quotients are isometric to a space $(N \times P)/\Gamma$, where $N$ is a symmetric space, $P$ is a manifold without conjugate points, and $\Gamma$ is a discrete subgroup of isometries. Consider, for instance, the parallel foliation $\mathcal F$ of a Euclidean space $V$ that is induced by an isoparametric submanifold $L$ of $V$. Such a foliation is a singular Riemannian foliation and is also called an \emph{isoparametric foliation}. In this case, the quotient $V/\mathcal F$ is isometric to $(p + \nu_p(L))/\Gamma$, where $p \in L$ is some point and $\Gamma$ is the finite Coxeter group generated by the reflections across the focal hyperplanes in $p + \nu_p(L) \subset V$. So $V/\mathcal F$ admits a flat orbifold covering which is a manifold, thus $\mathcal F$ is taut. In particular, our result implies that isoparametric submanifolds are taut, what is well known by \cite{HPT}. More generally, if a closed singular Riemannian foliation $\f$ on a Riemannian manifold $M$ is polar, i.e. through every regular point in $M$ there exists a complete submanifold meeting all the leaves and always perpendicularly, every section covers the orbit space $M/\f$ (as an orbifold). We therefore see again that the orbits of hyperpolar actions, i.e. when the sections are flat, are taut. Since sections are always totally geodesic and totally geodesic submanifolds of symmetric spaces are symmetric spaces, we also reobtain the result from \cite{BG} that a polar action of a compact Lie group on a compact rank-one symmetric space is taut. In \cite{GT} Gorodski and Thorbergsson classified all taut irreducible representations of compact Lie groups as either hyperpolar and hence equivalent to the isotropy representation of a symmetric space or as one of the exceptional representations of cohomogeneity three. Let $\rho: G \rightarrow \textbf{O}(V)$ be an exeptional representation, i.e. the induced action of $G$ on $V$ has cohomogeneity equal to three. Then, the restriction of this action on the unit sphere $S \subset V$ has cohomogeneity two, so that $S/G$ is isometric to a quotient $S^2/\Gamma$ of the round 2-sphere with a finite Coxeter group $\Gamma$. Since the orbits of the $G$-action on $V$ are taut if and only if the orbits of the $G$-action on $S$ are taut, it follows from Theorem D again that the exeptional representations are taut.\\

Practically, the only way to prove that a given submanifold $L \hookrightarrow M$ is taut is the explicit construction of so called \emph{linking cycles} for the energy. Namely, one has to find cycles that complete the local unstable manifolds associated to some Morse chart around the critical points below the corresponding critical energy. This concept is explained in Section \ref{sec:LC}.
For the proof of Theorem A in Section \ref{sec:Des} (cf. Theorem \ref{thm:A}) we therefore first make the easy observation that all the energy functionals $E_q : \mathcal P(M,L \times q) \rightarrow \mathbb R$ are Morse-Bott if and only if the normal exponential map $\exp^{\perp}:\nu(L) \rightarrow M$ has \emph{integrable fibers}, by what we mean that $(\exp^{\perp})^{-1}(\exp^{\perp}(v))$ is a union of submanifolds for all $v \in \nu(L)$. If so, we explicitely construct linking cycles for non-degenerate critical points, i.e. a basis for the (co-)homology of $\mathcal P(M,L \times q)$ if $q \in M$ is not a focal point, proving that $L$ is taut. For this purpose, for a normal vector $v \in \nu(L)$, we define $Z_v$ to be the set of all piecewise continuous paths $c :[0,1] \rightarrow \nu(L)$ obtained as follows. Follow the segment $t v$ towards the zero section up to the first focal vector $t_1v$, then take an arbitrary normal vector $v_1$ in the fiber of $\exp^{\perp}$ through $t_1v$ and follow the segment $tv_1$ towards the zero section up to first focal vector $t_2v_1$, then take an arbitrary normal vector in the fiber through $t_2v_1$ and repeat this procedure. By construction, for every $c \in Z_v$, $\exp^{\perp} \circ c$ is a broken geodesic from $\exp^{\perp}(v)$ to $L$ and we define the space $\Delta_v \subset \mathcal P(M,L \times \exp^{\perp}(v))$ to consist of all broken geodesic obtained in this manner reparameterized on $[0,1]$ after reversing the orientation. If the occurring focal multiplicities are locally constant, then $\Delta_v$ is an iterated fiber bundle and thus a compact manifold of the right dimension. In this case, its fundamental class ensures that it indeed provides a linking cycle. In the general case, $\Delta_v$ is a ``fiber bundle'' with singularities. Using sheaf cohomology we prove that it still carries a ``fundamental class'' and therefore actually represents a linking cycle. As mentioned above, Theorem B is then a direct consequence of this result.

In the second section we recapitulate the notion of singular Riemannian foliations and make some preliminary observations about taut foliations. As the main property, a geodesic either meets the leaves of a singular Riemannian foliation orthogonally at all or at none of its points. If a geodesic intersects one and hence all leaves perpendicularly, it is called \emph{horizontal}. Roughly speaking, the possibilities to vary a horizontal geodesic through horizontal geodesics consist of variations of the projection of the geodesic to the quotient and of variations through horizontal geodesics all of which meeting the same leaves simultaneously. This results in an index splitting for horizontal geodesics into \emph{horizontal} and \emph{vertical} index that we discuss in Section \ref{sec:HGI}, the latter one counting the intersections with the \emph{singular leaves} (with their multiplicities). In \ref{sec:PQ} we then prove Theorem C using Theorem A and the fact that the horizontal index is an intrinsic notion of the quotient (cf. Theorem \ref{thm:B}). Using the arguments from our proof of Theorem C, we are able to give a general  construction to obtain lots of examples of taut submanifolds including all the known examples that occur in families. Finally, Theorem D is then proved as a special case of Theorem C (cf. Theorem \ref{thm:infpol}).

At the end of this section we recall some basic facts about \emph{infinitesimally polar foliations}, i.e. those foliations whose quotients are orbifolds, from \cite{LT} and reformulate our Theorem D for those foliations. Infinitesimally polar foliations can also be described as those foliations admitting a \emph{canonical geometric resolution} (cf. \cite{L}), that is to say a canonically related regular foliation with an isometric quotient. As a consequence of our results we therefore observe that the canonical resolution of an infinitesimally polar foliation is taut if and only if the foliation is taut.\\

Our proof of Theorem A (as well as of Theorem C) may be viewed as a generalization of the construction of Bott and Samelson in \cite{BS} that proves that the orbit foliation of a variationally complete action, i.e. when the focal points of the orbits are only caused by singular orbits, is taut. Given an orbit of such an action, Bott and Samelson came up with concrete cycles associated to the critical points of the energy on the space of paths to a fixed point which, generically, represent a basis for the $\mathbb Z_2$-(co-)homology of the corresponding path space. For such a generic critical point $c$ their cycle can be described as a connected component of the set of broken horizontal geodesics, i.e. broken geodesics that intersect all the orbits orthogonally, that have the same projection to the orbit space as $c$, hence coincide with our space $\Delta(c)$ from Lemma \ref{lem:fiber}. Thus we reobtain their result as a special case of Corollary \ref{cor:fiber}. \\

Finally, we must warn the reader that the use of the term ``taut foliation'' could lead to confusions. In the theory of (regular) foliations there are other definitions of tautness, such as geometrically or topologically taut foliations. But in this work, by a taut foliation, we always mean a singular Riemannian foliation all of whose leaves are taut submanifolds as defined above.\\

\noindent \textbf{Acknowledgement}
This article contains the main results of the author's PhD thesis written at the Universit\"at zu K\"oln, Germany. The author would like to thank his teacher and supervisor Prof. Gudlaugur Thorbergsson for his long-term guidance. 

The author would also like to express his gratitude to Alexander Lytchak for his suggestion of the topic of taut foliations and plenty of instructive discussions as well as for his support and many hints and comments on this work. 

The financial support of the graduate school 1269 ``Global Structures in Geometry and Analysis'' of the DFG at the University of Cologne is gratefully acknowledged.

\newpage

\section{Taut Submanifolds}
\label{ch:TS}

\subsection{Linking Cycles}
\label{sec:LC}

The terminology of tautness for submanifolds of a Euclidean space was introduced by Carter and West in \cite{CW}. They call a submanifold $L$ of a Euclidean space $V$ \emph{taut} if there exists a field $\F$ such that for generic points $q \in V$, the squared distance functions $d_q^2 : L \To \R$, given by $d_q^2(p) = \|p-q\|^2$, are perfect with respect to the field $\F$. A definition similar to this can be used for submanifolds of the round sphere $S^n \subset V$.

\begin{red}
A Morse function on a complete Hilbert manifold $P$ is a smooth function $f: P \To \R$ which is bounded below, has a discrete critical set $\mr{Crit}(f)$ and satisfies Condition (C), i.e. if $(p_n)$ is a sequence of points in $P$ with $\left\{ f(p_n) \right\}$ bounded and $\|df_{p_n}\| \To 0$, then $(p_n)$ has a convergent subsequence.
Thus Condition (C) can be regarded as an analogue of a compactness claim in the infinite-dimensional setting.\par
If $p$ is a critical point for the Morse function $f$, then the index $\mr{ind}(p)$ is defined to be the dimension of a maximal subspace of $T_pP$ on which the Hessian is negative definite, i.e. the number of independent directions in which $f$ decreases. As in the finite dimensional case a Morse function gives rise to a cell complex with one cell of dimension $k$ for each critical point with index $k$, which is homotopy equivalent to $P$. If we set $P^r = \left\{p \in P \mbox{$|$} f(p) \leq r \right\}$, then the weak Morse inequalities say that if $\nu_k(a,b)$ denotes the number of critical points of index $k$ in $f^{-1}(a,b)$ for regular values $a < b$, then $b_k(P^b,P^a;\mathbb F) \leq \nu_k(a,b)$ for all $k$, where $b_k(P^b,P^a;\mathbb F)$ is the $k$-th Betti number of $(P^b,P^a)$ with respect to the field $\mathbb F$ and $f$ is called \emph{perfect} (\emph{with respect to} $\mathbb F$) if the weak Morse inequalities are equalities for all $k$ and all regular values $a < b$. For a detailed background we refer the reader to Part II of \cite{PT}.\\
A Morse-Bott function $f: P \To \R$ on a complete Hilbert manifold $P$ is a smooth function whose critical set is the union of closed submanifolds and whose Hessian is non-degenerate in the normal direction. That is to say, every critical point lies in a closed submanifold whose tangent space coincides with the kernel of the Hessian at each point. If so, the index of a critical point is defined to be the index of the restriction of the Hessian to the normal space of the corresponding critical manifold. Since the Hessian depends continuously on the points of the critical manifolds, the index is constant along the connected components of the critical set.     
\end{red} 

Using different approaches, Grove and Halperin \cite{GH} as well as Terng and Thorbergsson \cite{TT} defined a general notion of taut immersions into a complete Riemannian manifold. In \cite{TT} it has been proven that for submanifolds of a Euclidean space the generalized definition of tautness coincides with the one previously known. We are going to introduce this generalized notion using the exposition in \cite{TT}.\\

Let $(M,g)$ be a complete Riemannian manifold and let $H^1(I,M)$ denote the complete Riemannian Hilbert manifold of $H^1$ paths $I = [0,1] \To M$ with its canonical differentiable and metric structure, i.e. $H^1(I,M)$ is locally modeled on $H^1(I,\R^n)$. Recall that a path is of class $H^1$ if and only if it is absolutely continuous with finite energy. The charts for $H^1(I,M)$ are given by the identification of an open neighborhood of the zero section in $H^1(I,c^*(TM))$ , for $c \in H^1(I,M)$ piecewise smooth, with an open neighborhood of $c$ in $H^1(I,M)$ by the exponential map of $M$. Then, for such $c$, one has $T_cH^1(I,M) \cong H^1(I,c^*(TM))$ and the expression
\begin{equation*}
\langle X,Y \rangle = \int_Ig(X(t),Y(t)) dt + \int_Ig(\nabla X(t),\nabla Y(t))dt ,
\end{equation*}
which is a priori defined for $X,Y$ piecewise smooth can be extended to a complete Riemannian metric on $H^1(I,M)$ (cf. \cite{K}). Furnished with this differentiable structure, the map $e : H^1(I,M) \To M \times M$, given by $e(c) = (c(0),c(1))$, defines a submersion. 

Now for a proper immersion $\phi: L \To M$ into a complete Riemannian manifold and a point $q \in M$, we define the path space $\mathcal P_{(\phi,q)}(M,L)$ to be the pullback of $H^1(I,M)$ along the map $p \MTo (\phi(p),q)$ from $L$ into $M \times M$, i.e. $\mathcal P_{(\phi,q)}(M,L)$ consists of pairs $(p,c) \in L \times H^1(I,M)$ with $\phi(p) = c(0)$ and $c(1) = q$. In particular, $\mathcal P_{(\phi,q)}(M,L)$ inherits a smooth structure that turns it into a complete Hilbert manifold and one can show that the induced energy functional $E_{(\phi,q)}: \mathcal P_{(\phi,q)}(M,L) \To \R$, defined by 
\begin{equation*}
E_{(\phi,q)}((p,c)) = \int_I\|\dot c(t)\|^2dt ,
\end{equation*}
is a Morse function if and only if $q$ is not a focal point of $L$ (along any normal geodesic). Recall that for a normal vector $v \in \nu(L)$, the point $\exp^{\perp}(v) \in M$ is called a \emph{multiplicity $\mu$ focal point} of $L$ along the geodesic $\exp^{\perp}(tv)$ and $v$ is called a \emph{multiplicity $\mu$ focal vector} iff $\mr{dim}(\mr{ker}(d\exp^{\perp}_v)) = \mu >0$, where $\exp^{\perp}: \nu(L) \to M$ denotes the normal exponential map. The critical points of $E_{(\phi,q)}$ are exactly the pairs $(p,\gamma)$, where $\gamma$ is a geodesic, which starts perpendicularly to $L$ and ends in $q$. By the famous theorem of Morse, the index of a critical point $(p,\gamma)$ is then given by the sum 
\begin{equation*}
\mr{ind}((p,\gamma)) = \sum_{t \in (0,1)}\mu(t)
\end{equation*}
 over the multiplicities $\mu(t)$ of the points $\gamma(t)$ as focal points of $L$ along $\gamma$. 
 
A geodesic $\gamma$ that starts perpendicularly to $L$, i.e. $\dot{\gamma}(0) \in \nu(L)$, is called an $L$-\emph{geodesic}. In this case we frequently use the notation $\gamma_v$ to denote the $L$-geodesic $t \mapsto \exp^{\perp}(tv)$ (or respective restrictions). A vector field $J$ along an $L$-geodesic $\gamma$ is called an $L$-\emph{Jacobi field} along $\gamma$ iff it is a variational vector field of a variation of $\gamma$ through $L$-geodesics and the nullity of a critical point $(p,\gamma_v)$ of $E_{(\phi,\gamma_v(1))}$ is given by $\mu(v)= \mr{dim}(\mr{ker}(d\exp^{\perp}_v))$ and equals the dimension of the vector space 
\begin{equation*}
\left\{J \mbox{$|$} J \text{ is an } L-\text{Jacobi field along } \gamma_v, J(1) = 0\right\}.
\end{equation*}
In our setting of path spaces every energy sublevel contains finite dimensional submanifolds consisting of broken geodesics, each of them homotopy equivalent to this sublevel, such that the restriction of the energy to each of these submanifolds has the same relevant behavior, i.e. the critical points of the restriction are exactly the critical points of the energy functional and their indices and nullities coincide. In particular, the indices and nullities are finite.

For these facts and a detailed discussion on path spaces and the energy functional we refer the reader who is not familiar with these notions to \cite{K} and \cite{S}.\\

Finally, if $\phi$ is a closed embedding identifying $L$ with its image $\phi(L) \subset M$, we will drop the reference to the map $\phi$ and simply write $\mathcal P(M,L \times q)$ instead of $\mathcal P_{(\phi,q)}(M,L)$ for the space of $H^1$-paths from $L$ to $q$.\\

Suppose that $M = V$ is a Euclidean space and that, for simplicity, $L \subset V$ is a closed submanifold. Then, for every fixed point $q \in V$, the submanifold $L$ is diffeomorphic to the space of segments $\mathcal S = \left\{s_p(t) = p + t(q-p), p \in L\right\}$ in $\mathcal P(V, L \times q)$ and with respect to this identification the energy is given by the squared distance, i.e. $E_q(s_p) = d_q^2(p)= \|p-q\|^2$. Moreover, it is not hard to see that this identification respects the critical behavior of these functions. Since all the critical points of the energy $E_q$ are contained in $\mathcal S$ and, by convexity, the space $\mathcal P(V,L \times q)$ contains $\mathcal S$ as a strong deformation retract, the squared distance function $d_q^2$ is a perfect Morse function on $L$ if and only if $E_q$ is a perfect Morse function on $\mathcal P(V,L \times q)$. This observation led Terng and Thorbergsson to a natural generalization of the notion of a taut immersion into any complete Riemannian manifold $M$.\\

\begin{dfn}
A proper immersion $\phi : L \To M$ of a manifold $L$ into a complete Riemannian manifold $(M,g)$ is called \emph{taut} if there exists a field $\F$ such that the energy functional $E_{(\phi,q)} : \mathcal P_{(\phi,q)}(N,L) \To \R$, given by $E_{(\phi,q)}(p,c) = \int_I\|\dot c(t)\|^2dt$, is a perfect Morse function with respect to the field $\mathbb F$ for every point $q \in M$ that is not a focal point of $L$. In particular, a point $p \in M$ is called a \emph{taut point} if $\{p\}$ is a taut submanifold of $M$, i.e. $E_q : \mathcal P(M,p \times q) \To \R$ is perfect with respect to some field for every $q \in M$ that is not conjugate to $p$ along some geodesic. If a submanifold $L$ is taut and $\F$ is a field as in the definition of tautness, then $L$ is also called $\F$-\emph{taut}.
\end{dfn}

In \cite{Lei} Leitschkis called a manifold with only taut points \emph{pointwise taut} and we will continue with this notion.

\begin{nte}
In \cite{TT} it is shown that a properly immersed, taut submanifold of a simply connected, complete Riemannian manifold is actually embedded. Because we will see in Section \ref{sec:SP} that one can always assume that the ambient space is simply connected, we will proceed assuming that all submanifolds are embedded and closed, but all of our results will also hold in the case of a proper immersion. For this reason, if not otherwise stated, by a submanifold $L$ of $M$ we always mean an embedded submanifold and consider all submanifolds as subsets of $M$. Finally, a manifold is always assumed to be connected.
\end{nte}

The only way that is known to prove tautness in general, i.e. that a given Morse function is perfect, is the concept of \emph{linking cycles}, that we are going to explain now. For this reason let $f: P \To \R$ be again a Morse function on a complete Hilbert manifold. Then, for every $r \in \R$, the sublevel $P^r$ contains only a finite number of critical points of $f$ and we can assume that these critical points have pairwise distinct critical values. That the latter assumption is not restrictive follows from the fact that one can lift a small neighborhood of a critical point a little without changing the relevant behavior of the function. Moreover, using the flow of the negative gradient, one sees that for small $\varepsilon$ the sublevel sets $P^{r+\varepsilon}$ and $P^{r-\varepsilon}$ have the same homotopy type unless $r$ is a critical value. If so, let $p$ be the critical point of $f$ with $f(p)= r$ and choose an $\varepsilon$ such that $(r-\varepsilon,r+\varepsilon)$ contains no critical value except $r$. If we denote the index of $p$ by $i$ then $P^{r+\varepsilon}$ has the homotopy type of $P^{r-\varepsilon}$ with an $i$-cell $e_i$ attached to $f^{-1}(r-\varepsilon)$. Consider the following part of the long exact cohomology sequence of the pair of spaces $(P^{r+\varepsilon},P^{r-\varepsilon})$ with coefficients in a field $\mathbb F$:

\begin{align*}
\cdots & \To \underbrace {H^{i-1}(P^{r+\varepsilon},P^{r-\varepsilon})}_{= 0} \To  H^{i-1}(P^{r+\varepsilon}) \To H^{i-1}(P^{r-\varepsilon}) \xrightarrow{\partial^*} \underbrace {H^i(P^{r+\varepsilon},P^{r-\varepsilon})}_{\cong \mathbb F} \\
& \To H^i(P^{r+\varepsilon}) \To H^i(P^{r-\varepsilon}) \To \underbrace {H^{i+1}(P^{r+\varepsilon},P^{r-\varepsilon})}_{= 0} \To \cdots
\end{align*}

\noindent Since we are using coefficients from a field, we can switch between the more common homological and the, for our approach, more suitable cohomological point of view by dualization, i.e. $H^*(P^r;\mathbb F) \cong \mr{Hom}_{\mathbb F}(H_*(P^r;\mathbb F),\mathbb F)$. Anyway, we see that by passing from $P^{r-\varepsilon}$ to $P^{r+\varepsilon}$ the only possible changes in homology or cohomology occur in dimensions $i-1$ and $i$. To understand this geometrically let us have a look what happens in homology. In the first case, the boundary $\partial e_i$ of the attaching cell is an $(i-1)$-sphere in $P^{r-\varepsilon}$ which does not bound a chain in $P^{r-\varepsilon}$, i.e. $e_i$ has as boundary the nontrivial cycle $\partial e_i$ and so $\partial_* \neq 0$. In the second case, $\partial e_i$ does bound a chain in $P^{r-\varepsilon}$, which we can cap with $e_i$ to create a new nontrivial homology class in $P^{r+\varepsilon}$, that is to say $\partial_* = 0$ and $H_i(P^{r+\varepsilon}) \cong H_i(P^{r-\varepsilon}) \oplus \mathbb F$.

We see that the Morse inequalities are equalities if and only if
\begin{equation*}
H^i(P^{r+\varepsilon},P^{r-\varepsilon}) \To H^i(P^{r + \varepsilon}) \text{ is nontrivial, i.e. } \partial^* \equiv 0,
\end{equation*}
\noindent or, equivalently, 
\begin{equation*}
H_i(P^{r + \varepsilon}) \To H_i(P^{r+\varepsilon},P^{r-\varepsilon}) \text{ is surjective, i.e. } \partial_* \equiv 0,
\end{equation*}

\noindent for all critical points $p$ of $f$.\\

For a critical ponit $p \in P$, one can show that 
\begin{equation*}
H^*(P^{f(p) + \varepsilon},P^{f(p)-\varepsilon}) \cong H^*(P^{f(p)},P^{f(p)}\setminus \left\{p\right\}).
\end{equation*}
\noindent Thus suppose that we have a map $h_p : \Delta_p \To P^{f(p)}$ for every critical point $p$ such that the composition
\begin{equation*}
H^i(P^{f(p)},P^{f(p)}\setminus \{p\}) \xrightarrow{h_p^*} H^i(\Delta_p,h_p^{-1}(P^{f(p)}\setminus \{p\})) \To H^i(\Delta_p)
\end{equation*}
is nontrivial. In this case, because the connecting homomorphism $\partial^*$ is a natural transformation, we have that 
\begin{equation*}
h_p^* \circ \partial^* = \partial^* \circ h_p^*
\end{equation*}
\noindent and we conclude that the map $H^i(P^{f(p)},P^{f(p)}\setminus \{p\}) \To H^i(P^{f(p)})$ cannot be zero, so that $f$ is a perfect Morse function under this assumption. If so, we call the pair $(\Delta_p,h_p)$ a \emph{linking cycle for $p$}, the critical point $p$ \emph{of linking type}, and we say that the function $f: P \To \R$ is of \emph{linking type} if all the critical points are of linking type. Of course, if $f$ is perfect, then the inclusions of the corresponding sublevels define linking cycles, so that $f$ is of linking type. Thus we see that a Morse function is perfect if and only if it is of linking type.

\begin{nte}
At the end of this section we will prove that an $\F$-taut submanifold is always $\Z_2$-taut. By this reason, and due to the fact that dealing with (co-)homology there is just a little chance to get general results with other coefficients, we restrict our attention to the case $\F = \Z_2$. From now on, saying taut we always mean $\Z_2$-taut and we drop the reference to the field everywhere.
\end{nte} 

\begin{rem}
Using finite-dimensional approximations of the path space we see that in the setting we are interested in, singular cohomology is isomorphic to \v{C}ech cohomology (cf. Section \ref{sec:MT}). Because the latter groups satisfy a continuity property and are more easy to handle, we focus on the \v{C}hech cohomology groups in the following.
\end{rem}  

\subsection{The Main Tool}
\label{sec:MT}

As already mentioned in the introduction, the problem when dealing with general cycle constructions is the behavior of the focal data. The construction of Bott and Samelson works well, also in the general case, if the focal points along a variation do not collapse, i.e. if the cardinality of the intersections of normal geodesics with the focal set is locally constant. Unfortunately, the occurrence of focal collapses along a variation of normal geodesics cannot be avoided in general, but since these collapses depend continuously on the initial directions of the geodesics, it turns out that this indeed constitute no problem for our goal. In the following will work this out using the theory of sheaves and sheaf cohomology as it is presented in \cite{B1} and chapter 5 of \cite{Wa2}. We refer the reader to these textbooks for the facts that we presume.

To fix a notation we recall the following

\begin{dfn}
A \emph{sheaf (of Abelian groups)} on $X$ is a pair $(\mathcal A,\pi)$, where
\begin{enumerate}
\item $\mathcal A$ is a topological space;
\item $\pi: \mathcal A \To X$ is a local homeomorphism;
\item Each fiber $\mathcal A_x = \pi^{-1}(x)$ is an Abelian group and is called the \emph{stalk} of $\mathcal A$ at the point $x$;
\item The group operations are continuous, i.e. the map
\[
\begin{array}{ccl}
 \mathcal A \times_{\pi}\mathcal A & \To & \mathcal A , \\
 (\alpha,\beta) & \MTo & \alpha - \beta
\end{array}
\]
\noindent with $\mathcal A \times_{\pi}\mathcal A = \left\{(\alpha,\beta) \in \mathcal A \times \mathcal A \mbox{$|$} \pi(\alpha) = \pi(\beta)\right\}$, is continuous.
\end{enumerate}
If the context is clear we will drop the reference to the map $\pi$ and talk about the sheaf $\mathcal A$. For an Abelian group $G$, we say that $\mathcal A$ is a $G$-sheaf on $X$ if all the fibers $\pi^{-1}(x)$ are isomorphic to $G$.
\end{dfn}

Given a continuous map $f: X \To B$ from a compact Hausdorff topological space $X$ onto some nice space $B$ with a fundamental class in $\Z_2$-(co-)homology, e.g. a manifold, such that all the fibers are compact manifolds of constant dimension, one would expect that the union over the base $B$ of all the fibers defines a non-trivial class in $\Z_2$-(co-)homology in dimension equal to the sum of the (co-)homological dimension of $B$ and the fiber dimension. In order to prove our main theorem later on, we are dependent on a tool like this, because we want to construct explicit (co-)cycles with specified cohomological behaviour. The easiest way, known to the author, to prove such a statement is by means of sheaf cohomology.\\

The first step in this direction is the following easy observation for which we have not found any reference in the literature.

\begin{lem}
\label{lem:sheaf}
Let $X$ be a locally compact Hausdorff topological space. Then there are no nontrivial $\Z_2$-sheaves on $X$.
\end{lem}

\begin{proof}
Assume that $(\mathcal A,\pi)$ is a $\Z_2$-sheaf on $X$, i.e. $\pi^{-1}(x) \cong \Z_2$ for all $x \in X$. Then each fiber $\pi^{-1}(x)$ has exactly one element that is not zero and we denote this nontrivial element by $1_x$. Because the zero section $0 : X \To \mathcal A$ is a global section, it remains to observe that the well-defined map $1: X \To \mathcal A$, given by $1(x) = 1_x$, is continuous, what is straightforward. Thus the map $X \times \Z_2 \To \mathcal A$, given by $(x,\varepsilon) \MTo \varepsilon_x$ for $\varepsilon \in \{0,1\}$, defines an isomorphism of sheaves.  
\end{proof}

A \emph{presheaf} on a topological space $X$ is a contravariant functor from the category of open subsets of $X$, where the morphisms are just the inclusions, to the category of Abelian groups, i.e. a function that assigns to each open set $U$ an Abelian group $A(U)$ and to each pair $U \subset V$ a homomorphism, called the \emph{restriction}, $r_{U,V} : A(V) \To A(U)$ in such a way that $r_{U,U} = 1_{A(U)}$ and $r_{U,V} \circ r_{V,W} = r_{U,W}$ whenever $U \subset V \subset W$. The set of (local) sections of a sheaf $\mathcal A$ is a presheaf in the obvious way. Conversely, given a presheaf $A$, taking fiberwise the direct limit over all neighborhoods of a fixed point one gets a sheaf $\mathcal A$ by the so obtained set of germs topologized by the natural sections induced by $A$. In this case, one says that the sheaf $\mathcal A$ is \emph{generated} by the presheaf $A$.

\begin{exl} 
\begin{itemize}
\item For any standard cohomology theory $H^*$ on $X$, the assignment given by $U \MTo H^r(U;G)$ defines a presheaf, where the coefficients are taken to be any Abelian group $G$.
\item Let $f: Y \To X$ be a continuous map between topological spaces. Then for every $r \geq 0$ there is an associated presheaf on $X$ given by the prescription $U \To \check{H}^r(f^{-1}(U);G)$, where we denote by $\check H^*$ the \v{C}ech cohomology. The sheaf $\mathcal H^r(f;G)$ generated by this presheaf is called the \emph{Leray sheaf} of $f$ on $X$.
\end{itemize}
\end{exl}

There is a way to define cohomology theories with coefficients in a pesheaf or in a sheaf, what is the same for paracompact spaces (see, e.g. Chapter 6 of \cite{Sp}). But a developement of this theory would go beyond the scope of our discussion, the more so as it is not really necessary for our goal. By this reason we have to refer the reader to the literature, e.g. \cite{B1}, \cite{Sp} or \cite{Wa2}. Because it is all we need, we just want to mention that in the case of a constant sheaf $X \times G$ this cohomology is exactly the same as the usual \v{C}ech cohomology with coefficients in $G$. Moreover, it can be shown that if $X$ is a topological manifold of dimension $n$, then all the cohomology groups with coefficients in any sheaf vanish in dimensions greater than $n$. This is an important feature from which we make essential use of.

\begin{rem}
It is shown in \cite{Sp} that in the cases we are interested in, the cohomology groups $H^r(X;\mathcal G)$ with coefficients in the constant sheaf $\mathcal G = X \times G$ are nothing else than the Alexander-Spanier cohomology groups with coefficients in $G$ and that the latter coincide with the \v{C}ech cohomology groups $\check H^r(X;G)$. In particular, it is therefore clear that we have long exact sequences, excision, and that the homotopy axiom holds. Finally, because we are dealing only with nice spaces, please recall that for compact subsets $(K,L)$ of a manifold, the \v{C}ech cohomolgy groups $\check H^r(K,L)$ are isomorphic to the direct limit 
\begin{equation*}
\lim_{ \longrightarrow } \left\{H^r(U,V)\mbox{$|$}(K,L) \subset (U,V)\right\}
\end{equation*}
\noindent where the limit is taken over open subsets $(U,V) \supset (K,L)$.
\end{rem}

As mentioned above, the cohomology groups $H^r(U;\mathcal G)$ with coefficients in the constant sheaf $\mathcal G = X \times G$ are isomorphic with the corresponding \v{C}ech groups. Thus the generated sheaves are also isomorphic. Due to this, our definition of the Learay sheaf in this case behaves well with respect to the general definition of the Leray sheaf in the context of sheaf cohomology, which is, given a sheaf $\mathcal A$ on $Y$, generated by the presheaf $U \To H^r(f^{-1}(U);\mathcal A)$.

We now come to the heart of this section, that is a very powerful tool for our use. We formulate the following lemma in an easy to handle version, adjusted accordingly to our purpose, but the reader who goes through the proof will notice that it also holds under weaker assumptions, e.g. if $B$ as in the claim has the cohomological behavior of a manifold.

\begin{lem}
\label{lem:tool}
Let $X$ be a connected and compact Hausdorff topological space and let $f: X \To B$ be a continuous map onto a manifold $B$ of dimension $k$. Assume that every fiber $f^{-1}(b)$ is a connected manifold of (constant) dimension $n$ or, more general, that $\check H^n(f^{-1}(b);\Z_2) \cong \Z_2$ and $\check H^l(f^{-1}(b);\Z_2) = 0$ for all $l > n$, where $\check H^*$ denotes the \v Cech cohomology. Then $X$ has cohomological dimension $n+k$, i.e. $\check H^{n+k}(X;\Z_2) \cong \Z_2$ and $\check H^l(X;\Z_2) = 0$ for $l > n+k$.
\end{lem}

\begin{proof}
Since $X$ is compact and connected and $f$ is surjective, $B$ is compact and connected, too. If we consider cohomology with coefficients in the constant sheaf $\mathcal Z_2 = X \times \Z_2$, then, due to Theorem 6.1 in \cite{B1}, there exists a spectral sequence $\{E_r,d_r\}$ with $d_r: E_r^{m,l} \To E_r^{m+r,l-r+1}$ converging to $H^*(X;\mathcal Z_2)$ with $E_2$ page
\begin{equation*}
E_2^{m,l} = H^m(B;\mathcal H^l(f;\mathcal Z_2)) ,
\end{equation*}
\noindent where $\mathcal H^l(f;\mathcal Z_2)$ denotes the above defined Leray sheaf on $B$, generated by the presheaf $U \MTo H^l(f^{-1}(U);\mathcal Z_2)$. Further, it is also proven there that in our setting the stalks $\mathcal H^l(f;\mathcal Z_2)_p$ of the Leray sheaf are isomorphic to the cohomology groups of the corresponding fibers $H^l(f^{-1}(b);\mathcal Z_2) \cong \check H^l(f^{-1}(b);\Z_2)$. Then, by our assumptions and Lemma \ref{lem:sheaf}, the $n$-th Leray sheaf is the constant sheaf on $B$, that is to say $\mathcal H^n(f;\mathcal Z_2) \cong B \times \Z_2$. Therefore, the entry $E_2^{k,n}$ is just given by the $k$-th \v Cech cohomology group $\check H^k(B;\Z_2) \cong \Z_2$. Because $B$ is a manifold of dimension $k$, all the groups $H^m(B;\mathcal H^l(f;\mathcal Z_2))$ vanish for $m > k$, by dimensional reasons mentioned above. Also all the entries $E_2^{m,l}$ with $l > n$ vanish, because $\mathcal H^l(f;\mathcal Z_2)$ is the $0$-sheaf in this case. But this means that the entry $E_2^{k,n}$ survives in the spectral sequence since it is the top right entry in the nontrivial rectangle on the $E_2$ page. The second statement of the claim follows from the fact that if $m+l > k+n$ then $m > k$ or $l > n$. 
\end{proof} 

\subsection{An Equivalent Description}
\label{sec:Des}

Having achieved our key tool in the last section, we are now able to prove our main result.

\begin{thm}
\label{thm:A}
Let $L \subset M$ be a closed submanifold of a complete Riemannian manifold $M$. Then the following statements are equivalent.
\begin{enumerate}
\item $L$ is taut.
\item All energy functionals are Morse-Bott functions.
\item The fibers of the normal exponential map $\exp^{\perp}: \nu(L) \to M$ are integrable.
\end{enumerate}
\end{thm}

\begin{proof}
The implication $1. \Rightarrow 2.$ holds by Theorem 2.8 of \cite{TT}, where Terng and Thorbergsson show that if $L$ is a taut submanifold (with respect to any field), then all the energy functionals are Morse-Bott functions adapting the idea of Ozawa \cite{O} for the Euclidean case.

The equivalence of $2.$ and $3.$ is more or less by definition. If $S \subset \mathcal P(M,L \times q)$ is a critical submanifold of $E_q$, then the image of the map $S \to \nu(L), s \mapsto \dot s(0),$ is an integral manifold of the kernel distribution. Conversely, if $S \subset \nu(M)$ is an integral manifold of the kernel distribution, i.e. an open and compact submanifold of some fiber $(\exp^{\perp})^{-1}(q)$, then the image of the map $S \rightarrow \mathcal P(M,L \times q), v \mapsto \exp^{\perp}(tv)|_{[0,1]},$ defines a non-degenerate connected component of the critical set of $E_q$. 

Therefore, it remains to show $3. \Rightarrow 1.$ and we prove this by constructing explicit linking cycles for $\Z_2$-coefficients.

Let us first fix some notation before we discuss our construction. For convenience, let $\eta = \exp^{\perp}: \nu(L) \to M$ denote the normal exponential map. By the ray through $v \in \nu(L)$ we mean the linear map $r_v : \R_0^+ \To \nu(L)$, given by $r_v(t) = tv$, and by the segment to $v$ we mean $s_v = r_v|_{[0,1]}$. Then for $\alpha \geq 1$, the point $q = \eta(v)$ is a focal point of $L$ along the geodesic $\gamma_{\alpha v} = \eta \circ s_{\alpha v}$ if and only if $d\eta_v$ is singular and the multiplicity of $q$ as such a point is just the dimension of the kernel of $d\eta_v$. Recall that in the case where $d\eta_v$ is not onto we call $v$ a focal vector of multiplicity $\mu(v) = \mr{dim}(\mr{ker}(d\eta_v))$. Let us denote by $C$ the union of all points in $\nu(L)$ where $d\eta$ is singular and call it the (tangent) focal locus. We call every number in $r_v^{-1}(C)$ a focal time along the ray $r_v$. It is a well known fact that the focal times are discrete along any ray and that they depend continuously on the rays. Namely, that there are continuous functions 
\begin{equation*}
\lambda_i : S(1) = \left\{v \in \nu(L) \mbox{$:$} \|v\|^2 = 1\right\} \to \R
\end{equation*}
 with $0 < \lambda_1\leq \lambda_2\leq \dots$ and $r_v^{-1}(C) = \left\{ \lambda_i(v) \right\}_{i \geq 1}$ (cf. \cite{IT}). This implies that every vector $v \in \nu(L)$ has an open neighborhood $U$ such that every ray that intersects $U$ contains $\mu(v)$ focal vectors in $U$ counted with multiplicities, i.e. if $\mr{im}(r_w) \cap U \neq \emptyset$, then 
\begin{equation*}
\sum_{t \in r_w^{-1}(U)}\mu(tw) = \sum_{t \in r_w^{-1}(U \cap C)}\mu(tw) = \mu(v).
\end{equation*}
 Finally, we call a focal vector $v \in C$ regular if there is an open neighborhood $U$ of $v$ such that all rays which intersect $U$ intersect $U \cap C$ exactly once. Due to Warner \cite{Wa} and Hebda \cite{He}, the set $C_R$ of regular focal vectors is an open and dense subset of $C$ that is a codimension-one submanifold of $\nu(L)$ such that $T_v\nu(L) \cong T_vC_R \oplus \R v$ for all $v \in C_R$. Let $\nu(L)^R$ denote the set of vectors $v \in \nu(L)\setminus C$ such that $s_v$ intesects $C$ only in $C_R$, i.e. such that $w \mapsto \#\{s_w^{-1}(C)\}$ is constant on a neighborhood of $v$. Then $\nu(L)^R$ is obviously open in $\nu(L)$ and it is also dense. To see this, consider the function $n : \nu(L)\setminus C \To \Z$, given by 
\begin{equation*}
n(v) = \# \left\{s_v^{-1}(C)\right\},
\end{equation*}
\noindent which is lower semi-continuous by our above observations. Then 
\begin{eqnarray*}
\nu(L)^R & = & \left\{v \in \nu(L)\setminus C \mbox{$:$} \ n \text{ is constant on a neighborhood of } v\right\}\\
    & = & \left\{v \in \nu(L)\setminus C \mbox{$:$} \ s_v^{-1}(C) = s_v^{-1}(C_R)\right\}.
\end{eqnarray*}

\noindent Since the set of regular vectors $\nu(L) \setminus C$ is open and dense in $\nu(L)$, it is enough to show that $\nu(L)^R$ is dense in $\nu(L)\setminus C$. Thus assume that $\nu(L)^R$ is not dense in $\nu(L) \setminus C$. Then the complement of $\nu(L)^R$ in $\nu(L) \setminus C$ contains an open set $U$. The function $n$ admits its maximum $n_r$ on every intersection $U \cap B(r)$ of $U$ with an open tube $B(r)$ of radius $r$ around the zero section. Choose $r$ so large that $U \cap B(r) \neq \emptyset$. Due to the semi-continuity of $n$, the set $n^{-1}(n_r) \cap U \cap B(r)$ defines  an open subset of $\nu(L)\setminus C$ on which $n$ is constant, what clearly contradicts our definition of $\nu(L)^R$.\\
  
Suppose now that the singular kernel distribution $\bigcup_{v\in\nu(L)}\mr{ker}(d\eta_v)$ is completely integrable, i.e. through every point $v \in \nu(L)$ there is a $\mu(v)$-dimensional compact connected submanifold $C_v$ with $T_wC_v = \mr{ker}(d\eta_w)$ for all $w \in C_v$. Then we have that $C_v \subset S(\|v\|)$, where $S(\|v\|)$ denotes the sphere bundle over $L$ of normal vectors of length $\|v\|$, and the index $i(v) = \sum_{t \in (0,1)}\mu(tv)$ is constant along $C_v$. As above, for a vector $v \in S(1)$, we denote by $0 < \lambda_1(v) \leq \lambda_2(v) \leq \dots$ the (continuous) focal times along the ray $r_v$, counted with their multiplicities. 

Let us now define a function $m: \nu(L) \To [0,1)$ which assigns to a normal vector $v \in \nu(L) \setminus \{0\}$ the number 
\begin{equation*}
m(v) = \max \left\{t \in (0,1) \mbox{$|$} \mu(tv) \neq 0\right\}
\end{equation*}
\noindent if $i(v) > 0$, and $m(v) = 0$ if $i(v) = 0$ or if $v$ belongs to the zero section. In particular, by our above observations, the restriction of $m$ to each submanifold $C_v$ is continuous, because of $m(v) = \lambda_i(v/\|v\|)/\|v\|$ for some $i$ and $\|v\| \neq 0$. 

Denote by $\mathfrak C = \bigcup_{v \in \nu(L)}C_v$ the $\eta$-fiber decomposition of the normal bundle $\nu(L)$ and define $Q : \nu(L) \To \nu(L)/\mathfrak C$ to be the natural quotient map. Since the fibers of $Q$ are compact submanifolds of $\nu(L)$, which are connected components of the fibers of the continuous map $\eta$ from the complete space $\nu(L)$ to the manifold $M$, the quotient is a locally compact Hausdorff space and the restriction of the projection to $\nu(L) \setminus C$ is a homeomorphism onto an open subspace of $\nu(L)/\mathfrak C$. The fiber norms on $\nu(L)$ push down to the distance function from the image $Q(0)$ of the zero section, so that every compact subset in $\nu(L)/\mathfrak C$ has to be of bounded distance from $Q(0)$. In particular, the map $Q$ is proper and therefore closed.\\

We now define a natural cycle candidate $\Delta_v$ for every geodesic $\gamma_v = \eta \circ s_v$ with $v \in \nu(L)\setminus C$. Because $\eta$ factorizes over $\nu(L)/\mathfrak C$ by a map $\bar{\eta}: \nu(L)/\mathfrak C \To M$, we will work in the quotient space and consider the space $P(\nu(L)/\mathfrak C, Q(0) \times Q(v))$ of continuous paths $c : [0,1] \To \nu(L)/\mathfrak C$ from $Q(0)$ to $Q(v)$ with the compact open topology. The constructed cycles will embed in an energy preserving and obvious way into $\mathcal P(M,L \times \eta (v))$ under the map on the path space level induced by $\bar{\eta}$, so that we renounce the reference to the latter space for the rest of the proof.\\

Because the $\eta$-kernel distribution on $\nu(L)$ is integrable, there is a natural cycle through $Q \circ s_v$ intuitively having the right dimension. Namly, take $Z_v$ to be the set of all piecewise continuous maps from $[0,1]$ to $\nu(L)$ obtained (with the reversed orientation) as follows: Follow the segment $s_v$ towards the zero section up to the first focal vector  $m(v)v$, take a vector $w_1 \in C_{m(v)v}$ and follow the straight line $tw_1$ towards the zero section up to the first focal vector $m(w_1)w_1$. Then take an arbitrary focal vector $w_2$ in the corresponding leaf and follow the line $tw_2$ towards the zero section up to the first focal vector, then take an arbitrary focal vector $w_3$ in the corresponding leaf $C_{m(w_2)w_2}$ and follow the line $tw_3$ up to the first focal vector and so on. This process will end after a finite number of steps and we can push down these piecewise continuous maps via $Q$ obtaining broken geodesics $[0,1] \To \nu(L)/\mathfrak C$ starting in $Q(0)$ and ending in $Q(v)$ and we define $\Delta_v$ to be the injective image of $Z_v$ under this map. 

To be more precise, let us say that a tuple $c = (c_r,\dots,c_1)$ is an $\eta$-polygon on $[0,1]$ if there exists a partition $0 = t_r < t_{r-1} \dots < t_1 < t_0 = 1$ of the interval $[0,1]$ such that $c_i : [t_i,t_{i-1}] \To \nu(L)$ is given by $c_i(t) = r_{w_i}(t) =  tw_i$ for some vector $w_i \in \nu(L)$. Then for $v \in \nu(L)\setminus \{0\}$, let $Z_v$ be the set consisting of all $\eta$-polygons $c$ on $[0,1]$ inductively defined as follows. If $i(v) = \sum_{t\in (0,1)} \mu(tv) = 0$, set $Z_v = \{s_v\}$ and if $i(v) = 1$ with 
\begin{equation*}
t_1 = m(v) = s_v^{-1}(C) = \frac{\lambda_1(v/\|v\|)}{\|v\|},
\end{equation*}
define $Z_v$ to consist of pairs $c = (c_2(w),c_1)$, where $w \in C_{m(v)v}$ and 
\begin{equation*}
c_2(w):\left[0,m(v)\right] \to \nu(L) \text{ is given by } c_2(w)(t) = tw 
\end{equation*}
 and 
\begin{equation*}
c_1:\left[m(v),1\right] \to \nu(L) \text{ is given by } c_1(t) = tv.
\end{equation*}
Note that, because of $w \in C_{m(v)v}$, we have 
\begin{equation*}
\lambda_1(w/\|w\|) = \|w\| = m(v)\|v\| = \lambda_1(v/\|v\|).
\end{equation*}
 
Now assume that we already have defined $Z_w$ if $i(w) \leq n$ in such a way that it consists of all $\eta$-polygons $c = (c_r,\dots,c_1)$ with $c_i:[t_i,t_{i-1}] \To \nu(L)$ satisfying $c_{i+1}(t_i) \in C_{c_i(t_i)}$ for $i \geq 1$. Let $v$ be a normal vector with $i(v) = n+1$. Then for every vector $w$ in the fiber $C_{m(v)v}$ through $m(v)v$ we have $i(w) = i(m(v)v) < i(v)$ and we define the space $Z_v$ to consist of pairs $(d(w),c_1)$ for some $w \in C_{m(v)v}$ with $c_1 :[m(v),1] \to \nu(L)$, defined by $c_1(t) = tv$ and $d(w) : [0,m(v)] \to \nu(L)$ is a linear reparameterization on $[0,m(v)]$ of an element $\tilde d \in Z_w$, i.e. there exists an $\eta$-polygon $\tilde d = (\tilde d_r,\dots,\tilde d_1) \in Z_w$ with $\tilde d_i$ defined on an interval $[\tilde t_i,\tilde t_{i-1}]$, for a partition $0=\tilde t_r < \tilde t_{r-1} < \dots < \tilde t_1< \tilde t_0=1$, and $d(w) = (d_r,\dots,d_1)$ with $d_i:[t_i,t_{i-1}] \To \nu(L)$, given by $d_i(t) = \tilde d_i\left(\frac{t}{m(v)}\right)$, where $t_i = m(v) \tilde t_i$. If we set $c_{i+1} = d_i$, we have that $c= (d(w),c_1) =(c_{r+1},\dots,c_2,c_1)$ with $c_1 = s_{v}|_{[m(v),1]}$ and for $i \geq 2$ and $w_1 = v$ we obtain the formula
\begin{equation*}
c_i(t) = s_{\frac{\|v\|}{\|w_i\|}w_i}(t), t \in \left[m(w_i)\frac{\|w_i\|}{\|v\|}, m(w_{i-1})\frac{\|w_{i-1}\|}{\|v\|}\right] \text{ with } w_i \in C_{m(w_{i-1})w_{i-1}}.
\end{equation*}
\noindent Thus in the above notation $t_i = m(w_i)\frac{\|w_i\|}{\|v\|}$. Note that this is well defined because of $m(w_{i-1})\|w_{i-1}\| = \|w_i\|$ and our definition that $m(w) = 0$ if $i(w) = 0$. We can regard an $\eta$-polygon on $[0,1]$ as a piecewise continuous map $c: [0,1] \To \nu(L)$, defined by $c(0) = 0$ and $c|_{(t_i,t_{i-1}]} = c_i|_{(t_i,t_{i-1}]}$, where, of course, here $c(0) = 0$ means the origin of the normal space that is uniquely defined by $c(\varepsilon)$ for some small number $\varepsilon > 0$. Anyway, there is a well-defined injective map 
\begin{equation*}
\bar Q : Z_v \To P(\nu(L)/\mathfrak C, Q(0) \times Q(v)),
\end{equation*}
\noindent given by $\bar Q(c)|_{[t_i,t_{i-1}]} = Q \circ c_i$ with energy     
\begin{eqnarray*}
E(\bar Q(c)) & = & (1-m(w_1))\|v\|^2 + \sum_{i \geq 2}E(c_i)\\
                  & = & (1-m(w_1))\|v\|^2 + \sum_{i \geq 2}(m(w_{i-1})\frac{\|w_{i-1}\|}{\|v\|} - m(w_i)\frac{\|w_i\|}{\|v\|})\|v\|^2\\
                  & = & \|v\|^2.
\end{eqnarray*}
\noindent Now we define the space $\Delta_v$ to be the image $\bar{Q}(Z_v) \subset P(\nu(L)/\mathfrak C, Q(0) \times Q(v))$ with the relative topology, i.e. induced by the compact open topology on the path space $P(\nu(L)/\mathfrak C, Q(0) \times Q(v))$. With this topology the space $\Delta_v$ is compact, because $Q$ is proper. The map $\bar Q : Z_v \To \Delta_v$ is a bijection and we topologize $Z_v$ by the postulation that this map is a homeomorphism. As mentioned above, we can regard $Z_v$ as a space of piecewise continuous maps. 

We follow this direction and define $e : Z_v \times [0,1] \To \nu(L)$ by $e(c,0) = c(0)$ and $e(c,t) = \lim_{t' \nearrow t}c_i(t')$ if $t \in (t_i,t_{i-1}]$, so that $t \MTo e_t(c)$ is the required map. Let us set $\bar e : \Delta_v \times [0,1] \To \nu(L)/\mathfrak C$ for the continuous evaluation map, given by the prescription $\bar e(\bar Q(c),t) = \bar Q(c)(t)$, and consider the commutative diagram

\[
\begin{xy}
\xymatrix
{
Z_v \times [0,1] \ar[r]^e \ar[d]^{\bar Q \times \mr{id}} & \nu(L) \ar[d]^Q\\
\Delta_v \times [0,1] \ar[r]^{\bar e} & \nu(L)/\mathfrak C\\
}
\end{xy}
\]  

\noindent from which it follows that $e$ is continuous in $(c,t)$ if $\bar Q(c)(t) \notin Q(C)$. If we define $e_t = e(\cdot,t)$, then $e_{m(v)} : Z_v \To \nu(L)$ is also continuous. To see this, we first observe that $e_{m(v)}$ is continuous iff it is continuous considered as a map to the submanifold $C_{m(v)v} \subset S(m(v)\|v\|)$. Thus take an open subset $U \subset C_{m(v)v}$ and note that by definition we have that $e_{m(v)}(c) = w \in U$ iff the image of $c_2|_{[m(w),m(v)]}$ is contained in the set $[m(\cdot),1]\cdot U = \left\{r(w)w \mbox{$|$} r(w) \in [m(w),1], w \in U\right\}$. Because $m$ is bounded away from $1$ on $C_{m(v)}$, we can find an $\varepsilon > 0$ such that $(1-\varepsilon,1)\cdot C_{m(v)} \subset \nu(L) \setminus C$ and because $Q$ is an embedding on $\nu(L) \setminus C$ we have
\begin{equation*}
e_{m(v)}^{-1}(U) = e_{m(v)-\varepsilon/2}^{-1}((1-\varepsilon,1)\cdot U) = \bar Q^{-1}(\bar e_{m(v)-\varepsilon/2}^{-1}(Q((1-\varepsilon,1)\cdot U))),
\end{equation*}
\noindent which is therefore open.\\

The crucial point with regard to our goal is that therefore the map 
\begin{eqnarray*}
pr_v = e_{m(v)} \circ \bar Q^{-1}  :  \Delta_v & \To & C_{m(v)v}\\
 \bar Q(c) & \mapsto & c(m(v))
\end{eqnarray*}
 is continuous, so $\Delta_v$ is given as the (continuous) family of fibers 
\begin{equation*}
\Delta_v = \bigcup_{w \in C_{m(v)v}}pr_v^{-1}(w) \cong \bigcup_{w \in C_{m(v)v}}\Delta_w,
\end{equation*} 
 what enables us to use an inductive argument to verify the right cohomological behavior. For this reason, we identify $pr_v^{-1}(w) \cong \Delta_w$ by restriction and reparametrization, i.e. forgetting the last irrelevant segment we regard a broken geodesic $\bar Q(c)$ in $\Delta_v$ as a path from $Q(0)$ to the furthermost breaking point $Q(C_{m(v)v})$, namely as a path in the space $\Delta_{c(m(v))}$.

By definition, $\Delta_v = \{ s_v \}$ if $i(v) =0$. If $i(v) = 1$, then $C_{m(v)v} \cong S^1$ and we have $pr_v^{-1}(w) \cong \Delta_w \cong \{ s_w \}$ for all $w \in C_{m(v)v}$. Of course, in this case we have $\check H^1(\Delta_v) \cong \Z_2$ and $\check H^k(\Delta_v) = 0$ if $k > 1$. In the general case, we note again that for all $w \in C_{m(v)v}$ we have 
\begin{equation*}
i(w) = i(m(v)v) = i(v) - \mu(m(v)v) = i(v) - \mathrm{dim}(C_{m(v)v}).
\end{equation*}
In particular, if $i(v) > 0$, then $i(w) < i(v)$ for all $w \in C_{m(v)v}$. Thus, because of the fact that $pr_v^{-1}(w) \cong\Delta_w$, we can assume by induction that we have $\check H^{i(w)}(pr_v^{-1}(w)) \cong \mathbb Z_2$ and $\check H^k(pr_v^{-1}(w)) = 0$ if $k > i(w)$ for all $w \in C_{m(v)v}$. Applying Lemma \ref{lem:tool} to the map $pr_v : \Delta_v \To C_{m(v)}$ it then follows  that $\check H^{i(v)}(\Delta_v) \cong \mathbb Z_2$ and $\check H^k(\Delta_v) = 0$ if $k > i(v)$. \\

Coming so far, it remains to prove that the spaces $\Delta_v$ indeed represent linking cycles for generic geodesics $Q \circ s_v$, because using continuity arguments as in Section \ref{sec:SRB} this would imply tautness. But this follows from the fact that, due to \cite{Wa1} and \cite{He}, $\mathfrak C$ defines a smooth distribution on every connected component of the set of regular focal vectors, so that it is more or less obvious by our construction that $Q \circ s_v$ admits a manifold neighborhood in $\Delta_v$ for all $v \in \nu(L)^R$. Further, this neighborhood can be deformed into the local unstable manifold in some Morse chart around $Q \circ s_v$, because of  the following expression of the tangent space that is a direct consequence of our construction. Namely, 
\begin{equation*}
T_{Q \circ s_v}\Delta_v = \bigoplus_{k = 1}^r\mathcal J(t_k),
\end{equation*}
 where $s_v^{-1}(C) = \{t_1,\dots,t_r\}$ and 
$\mathcal J(t_k)$ equals the vector space of continuous vector fields $J$ along $Q \circ s_v$ such that $J|_{[0,t_k]}$ is an $L$-Jacobi field along $Q \circ s_v$ and $J|_{[t_k,1]} \equiv 0$. A direct computation or a look at the proof of the Index Theorem of Morse (cf. \cite{S}) now shows that the projection of  $T_{Q \circ s_v}\Delta_v $ onto the tangent space at $Q \circ s_v$  of the unstable manifold corresponding to some Morse chart  is an isomorphism, that is to say locally around the critical point, $\Delta_v$ can be deformed into the unstable manifold of some Morse chart. In this case, if we denote by $\mathcal P_{L,\eta(v)}$ the space $\mathcal P (M,L \times \eta(v))$, the following commutative diagramm 

\[
\begin{xy}
\xymatrix
{
\check H^{i(v)}(\mathcal P_{L,\eta(v)}^{\|v\|^2}) \ar[r]  & \check H^{i(v)}(\Delta_v)\\
\check H^{i(v)}(\mathcal P_{L,\eta(v)}^{\|v\|^2},\mathcal P_{L,\eta(v)}^{\|v\|^2} \setminus \left\{\eta \circ s_v\right\}) \ar[u] \ar[r]^{ {}\cong} & \check H^{i(v)}(\Delta_v, \Delta_v \setminus \left\{Q \circ s_v\right\}) \ar[u]_{\cong}\\
}
\end{xy}
\]  

\noindent  yields the claim if $v \in \nu(L)^R$ is not a focal vector. Since for manifolds \v{C}ech cohomology is isomorphic to singular cohomology and $\nu(L)^R$ is dense in $\nu(L)$, we deduce, with the same arguments as in the proof of Proposition $2.7$ in \cite{TT}, that the energy $E_q : \mathcal P(M,L \times q) \To \R$ is $\Z_2$-perfect for all points $q$ that are not focal points of $L$.   

\end{proof}

\begin{rem}
As we mentioned in the last section, for compact subsets $(K,L)$ of a manifold $P$ the \v{C}ech cohomolgy groups $\check H^j(K,L)$ are isomorphic to the direct limit 
\begin{equation*}
\lim_{ \longrightarrow } \left\{H^j(U,V)\mbox{$|$}(K,L) \subset (U,V)\right\}
\end{equation*}
\noindent where the limit is taken over open subsets $(U,V) \supset (K,L)$. Therefore, one could also show directly that 
\begin{equation*}
H^{i(v)}(\mathcal P_{L,\eta(v)}^{\|v\|^2},\mathcal P_{L,\eta(v)}^{\|v\|^2} \setminus \left\{\eta \circ s_v\right\}) \To \check H^{i(v)}(\Delta_v, \Delta_v \setminus \left\{\eta \circ s_v\right\}) \To \check H^{i(v)}(\Delta_v)
\end{equation*}
is nontrivial for all the spaces $\Delta_v$ with $v \in \nu(L) \setminus C$, because using the deformation retraction of $P_{L,\eta(v)}^{\|v\|^2+\varepsilon}$ onto the Morse complex one can assume that a neighborhood base of $\eta \circ s_v$ in $\Delta_v$ is contained in some ball around the origin in $\R^{i(v)}$.
\end{rem}

As a direct consequence of the proof of Theorem \ref{thm:A} and the above remark, we obtain the following fact, which was so far not even known in the case of a Euclidean space.

\begin{thm}
\label{cor:MB2}
If a closed submanifold of a complete Riemannian manifold is taut with respect to some field, then it is also $\Z_2$-taut.
\end{thm}  

It is worth mentioning that although tautness is defined by means of perfect Morse functions an analogous statement for Morse functions is wrong. This indicates the high degree of geometry involved in this setting. 

\begin{exl}
Consider the unit $3$-sphere $S^3 \subset \mathbb C^2$ with the $\Z_p$-action, generated by 
\begin{equation*}
1\cdot(z,w) = (\mr e^{\frac{i2\pi}{p}}z,\mr e^{\frac{i2\pi q}{p}}w) 
\end{equation*}
with $p$ and $q$ relatively prime. The quotient $S^3/\Z_p$ is known as the \emph{lens space} $L(p,q)$ with fundamental group $\pi_1(L(p,q)) \cong \Z_p$ and $H_k(L(p,q);\Z_p) \cong \Z_p$ for $k=0,1,2,3$. Now, the Morse-Bott function $|z|^2$ is invariant under the $\Z_p$-action so it descends to a Morse-Bott function on $L(p,q)$ with critical set corresponding to the two critical circles $z=0$ and $w=0$ of index $0$ and index $2$, respectively. One can perturb this Morse-Bott function in the neighborhoods of the critical submanifolds by adding a bump function depending on the distance to the respective submanifold times a (perfect) Morse function on the respective circle so that the indices add. This therefore results in a $\Z_p$-perfect Morse function on $L(p,q)$. Now, if $p$ is odd, we have $H_2(L(p,q);\Z_2)= 0$, so that this $\Z_p$-perfect Morse function is not perfect with respect to $\Z_2$. In particular, there are no taut immersions of $L(2k+1,q)$ into a Euclidean space. Indeed, Thorbergsson actually proved in \cite{Tor} that there are no taut immersions of $L(p,q)$ into a Euclidean space except the case of the projective space $\R P^3 = L(2,1)$ by showing that the first nontrivial homology group can only have torsion elements of order two.  
\end{exl}

\section{Taut Foliations}
\label{ch:TF}

Even if there are not many examples of taut submanifolds, a remarkable observation is that they often occur, if at all, in families, which then decompose the ambient space. In this section we therefore focus on taut families as they usually occur, namely on singular Riemannian foliations all of whose leaves are taut. For this reason, we first recall some basic facts about singular Riemannian foliations and make some preliminary observations which we need to prove our second result in Subsection \ref{sec:PQ} that characterizes taut singular Riemannian foliations by means of their quotients.   

\subsection{Singular Riemannian Foliations and Orbifolds}
\label{sec:SRB}

For a more detailed discussion on singular Riemannian foliations and the proofs of the following statements we refer to \cite{Mol} and \cite{LT}.

\begin{dfn} 
Let $\f$ be a partition of a manifold $M^{n+k}$ into connected, injectively immersed submanifolds with maximal dimension $n$. For a point $p \in M$, let $L_p$ denote the element of $\f$ which contains $p$. Set 
\begin{equation*}
T\f = \bigcup_{p \in M} T_pL_p.
\end{equation*}
 Then the partition $\f$ is called a \emph{singular foliation of $M$ of dimension $n$/ codimension $k$} iff the $C^{\infty}(M)$-module $\Gamma(T\f)$ of smooth vector fields $X$ tangential to $\f$, i.e. with $X_p \in T_pL_p$ for all $p \in M$, exhaust $T_pL_p$ for every $p \in M$. We call the elements of $\f$ \emph{leaves}. A leaf is \emph{regular} if it has dimension $n$, otherwise \emph{singular}. A point belonging to a regular leaf is \emph{regular}, otherwise \emph{singular}. By $M_0$ we denote the set of regular points and call it the \emph{regular stratum}. If $(M,g)$ is a Riemannian manifold, a singular foliation is called a \emph{singular Riemannian foliation} if every geodesic in $M$ which intersects one leaf orthogonally intersects every leaf it meets orthogonally.  
\end{dfn}

We sometimes also speak about a singular Riemannian foliation $(M,\f)$, or also $(M,g,\f)$ if we want to abbreviate that $\f$ is a singular Riemannian foliation on the Riemannian manifold $M$, or $(M,g)$.

\begin{exl}
The set of orbits of an isometric Lie group action on a Riemannian manifold $M$ is a singular Riemannian foliation, closed if and only if the group considered as a subgroup of the isometry group is closed.
\end{exl}

For $d \leq n$, denote by $M_d$ the subset of all points $p \in M$ with fixed leaf dimension $\mr{dim}(L_p) = n-d$. Since the dimension of the leaves varies lower semi-continuously, the set $\bigcup_{d'\leq d}M_{n-d'} = \{p \in M \mbox{$|$} \mr{dim}(L_p)\leq d\}$ is closed. Further, $M_d$ is an embedded submanifold of $M$ and the restriction of $\f$ to $M_d$ is a (regular) Riemannian foliation. The main stratum $M_0$ is open, dense and connected if $M$ is connected. All the other singular strata have codimension at least $2$ in $M$. \\

Let $p$ be a point in $(M,g,\f)$ and let $B$ be a small open ball in $L_p$. Then there is a number $\varepsilon > 0$ and a \emph{distinguished tubular neighborhood} $U$ at $p$ so that  $U$ is the diffeomorphic image of the $\varepsilon$-disc bundle $\nu^{\varepsilon}(B)$ under the exponential map and $\pi^{-1}(q)$ is a global transversal of $(U,\f|_U)$ for all $q \in B$, i.e. it meets all the leaves of $\f|_U$ and always transversally, where we denote by $\pi$ the foot point projection. In this case, for each real number $\lambda \in [-1,1] \setminus \{0\}$ the map $h_{\lambda} : U \To U$, given by $h_{\lambda}(\mathrm{exp}(v)) = \exp(\lambda v)$, for all $v \in \nu^{\varepsilon}(B)$, preserves $\f$.

Indeed, the fact that the geodesics perpendicular to one leaf remain perpendicular to the leaves implies that the leaves of a singular Riemannian foliation are locally equidistant.

\begin{dfn}
We say that a singular Riemannian foliation $(M,\f)$ \emph{ has the property} $P$ if every leaf of $\f$ has the property $P$, e.g. $\f$ is closed if all the leaves are closed subspaces of $M$.
\end{dfn}

It is well kown that the leaves of a closed singular Riemannian foliation $\f$ on a complete Riemannian manifold $M$ admit global $\varepsilon$-tubes, so that the distance between two leaves is globally constant. In this case, the quotient $M/\f$ is a complete metric space, where the distance between two points is just the distance between the corresponding leaves as submanifolds of $M$.

\begin{lem}
\label{lem:prop}
Let $(M,\f)$ be a singular Riemannian foliation. Then a leaf $L \in \f$ is embedded if it is closed.
\end{lem}

\begin{proof}
Let $r$ be the dimension of $L$. Then $M_{n-r}$ is an embedded submanifold of $M$ and $\f|_{M_{n-r}}$ is a regular Riemannian foliation. Due to Molino (cf. p.22 of \cite{Mol}), the statement is true for $(M_{n-r},\f|_{M_{n-r}})$, so it is true for $(M,\f)$.
\end{proof} 

Again, let $p \in M$ be a point and let $B$ be a small open neighborhood in the leaf $L_p$ through $p$. Then there is  a distinguished tubular neighborhood $U$ around $p$ such that there is an embedding $\phi$ of $U$ into the tangent spaces $T_pM$ with $d\phi_p = \mathrm{Id}$ and a singular Riemannian foliation $\mathfrak{F}_p$ on $T_pM$, called \emph{infinitesimal singular Riemannian foliation of $\f$ at the point $p$} that coincides with $\phi_*\f$ on $\phi(U)$ and such that $\mathfrak{F}_p$ is invariant under the homotheties $r_{\lambda} : T_pM \To T_pM$, given by $r_{\lambda}(v) = \lambda v,$ for all $\lambda \neq 0$. If $\f$ is locally closed at $p$, the quotient $T_pM/\mf F_p$ is non-negatively curved in the sense of Alexandrov, and is the tangent space of $U/\f$ at the leaf $L \cap U \in U/\f$ and the local quotient $U/\f$ has curvature bounded below. The inclusion $U \To M$ induces a map between the quotients $U/\f \To M/\f$, which is open and finite-to-one if $\f$ is closed. So, assume that $\f$ is closed and $M$ is complete. Then the quotient $M/\f$ is a complete metric space with the metric induced by the distance of the leaves of $\f$ (as submanifolds). Let $T$ be a global $\varepsilon$-tube around $L$, where  $\varepsilon$ is chosen as in the definition of the distinguished neighborhood $U$ above, then $T$ is saturated, i.e. it is a union of leaves, and $T/\f$ is a neighborhood of $L$ in the global quotient $M/\f$. In this case, there is a finite group of isometries $\Gamma$ acting on the local quotient $U/\f$ that fixes the plaque $L \cap U \in U/\f$ such that $U/\f$ is isometric to $T/\f$.\\

A concept that is closely related to that of (Riemannian) foliations is the notion of (Riemannian) orbifolds. We therefore summarize the basic facts and defintions in a reminder and discuss orbifold coverings thereafter in more detail, because we will need this later on.

\begin{red}
As a manifold is locally modeled on open sets $U$ of $\R^n$ orbifolds are locally modeled on finite quotients $U/G$, where $G \subset \mr{Diff}(U)$ is a finite group. Thus an \emph{orbifold} $(B, \mathfrak U)$ is a Hausdorff topological space $B$ together with a maximal \emph{orbifold atlas} $\mathfrak U$ consisting of compatible orbifolds charts $(U,G,\varphi)$, i.e. the mapping $\varphi : U \to B$ induces a homeomorphism $U/G \cong \varphi(U)$ with $\varphi(U)$ open. The compatibility condition just means that the transition maps $(U_1,G_1,\varphi_1) \to (U_2,G_2,\varphi_2)$ commute with the respective projections. In this case there is (up to isomorphism) a well defined notion of \emph{isotropy group} for the points in $B$ as the isotropy group of any preimage point in an arbitrary chart. The points in $B$ with trivial isotropy group are called \emph{regular} otherwise \emph{singular}. If in addition $B$ is a metric space and there is a Riemannian metric on the $U_i$ such that $G_i \subset \mr{Iso}(U_i)$ and the homeomorphisms $U_i/G_i \to \varphi_i(U_i)$ are isometric, then $B$ is called a \emph{Riemannian orbifold}. Because the terms of Riemannian geometry are defined locally and are invariant under isometries, there is an obvious way to generalize the notions of Riemannian geometry to Riemannian orbifolds as they are locally isometric to finite quotients. See \cite{LT} for a more detailed discussion.  
\end{red}

\begin{dfn} 
A \emph{covering orbifold} of an orbifold $(B,\mathfrak{U})$ is an orbifold $(\widetilde{B},\widetilde{\mathfrak{U}})$, with a map $P: \widetilde{B} \To B$ between the underlying spaces such that each point $b \in B$ has a neighborhood $V = U/G$ for which each component $\widetilde V_i$ of $P^{-1}(V)$ is isomorphic to $U/G_i$ for some subgroup $G_i \subset G$ and such that the isomorphisms commute with the projection $P$.
\end{dfn}

It is well known that any orbifold admits a \emph{universal orbifold covering}, that is to say for a regular base point $b_0 \in B$, there exists a pointed connected covering orbifold $P: \widetilde{B} \To B$ with base point $\tilde{b}_0$ projecting to $b_0$ such that for any other covering orbifold $P': B' \To B$ with base point $b'_0$ and $P'(b'_0) = b_0$, there is a lift $Q: \widetilde B \To B'$ of $P$ along $P'$ to an orbifold covering. In particular, a universal orbifold covering is regular in the sense that its group of deck transformations acts simply transitive on a generic fiber. This group is denoted by $\pi_1^{orb}$ and is called the \emph{orbifold fundamental group}.  

\begin{dfn}
An orbifold is called \emph{good} if it is a global quotient or, equivalently, if the universal covering orbifold is a manifold, i.e. there are no singular points.
\end{dfn}

In the case of a Riemannian orbifold all the definitions are to be modified in the obvious manner, so that we can speak about \emph{Riemannian orbifold coverings} and \emph{good Riemannian orbifolds}. Of course, the statement about the universal orbifold covering also holds in the Riemannian category.

\begin{exl}
Again let $N$ be a Riemannian manifold and let $\Gamma$ be a discrete group of isometries of $N$. If $\widetilde N$ denotes the universal covering of $N$, then $\widetilde N$ is the universal Riemannian covering orbifold of $N/\Gamma$. This can be seen, for instance, by the observation that every covering orbifold of $N/\Gamma$ has to be of the form $\widetilde{N}/\widetilde{\Gamma}'$, where $\widetilde{\Gamma}'$ is a subgroup of the group $\widetilde{\Gamma}$ of deck transformations of $\widetilde N$ over $N/\Gamma$. Hence the two definitions of a good orbifold are indeed equivalent.
\end{exl}

Because we will use it in the following, we will formulate the next observation as a lemma.

\begin{lem}
\label{lem:orbicov}
Let $\f$ be a closed (regular) Riemannian foliation on a complete Riemannian manifold $M$ and let $\widetilde{\f}$ denote its lift to the universal Riemannian co\-vering $\widetilde M$ of $M$. Then the quotient $\widetilde M/\widetilde{\f}$ is a complete Riemannian manifold, i.e. $\widetilde{\f}$ is simple, if $M/\f$ is a good Riemannian orbifold. In particular, the orbifold covering $\widetilde M/\widetilde{\f} \To M/\f$ coincides with the universal Riemannian orbifold covering.
\end{lem}

\begin{proof}
Since $\f$ is closed its lift $\widetilde{\f}$ is closed too (cf. Lemma \ref{tlem:1}), and the leaves of $\widetilde{\f}$ admit global $\varepsilon$-tubes, because $\widetilde M$ is complete. Due to \cite{Hae} or \cite{Sal}, there is a surjective homomorphism $\pi_1(\widetilde M) \To \pi_1^{orb}(\widetilde M/\widetilde{\f})$, where the latter group is the group of deck transformations of the universal orbifold covering of $\widetilde M/\widetilde{\f}$. Now, if $M/\f$ is a good Riemannian orbifold, its branched cover $\widetilde M/\widetilde{\f}$ is a good Riemannian orbifold, too. But then $\pi_1^{orb}(\widetilde M/\widetilde{\f}) = 1$ implies that $\widetilde M/\widetilde{\f}$ already coincide with its universal covering orbifold and is therefore a manifold. 
\end{proof}

\subsection{Simplifications and Preliminaries}
\label{sec:SP} 

In this section we give some preliminary results to simplify the discussion of taut foliations, e.g. we will see that one can always assume that the manifold is simply connected (cf. Lemma \ref{tlem:2}) and that a dense family of leaves forces a foliation to be taut (cf. Corollary \ref{tlem:7}). 

\begin{dfn}
Let $\mathcal F$ be singular Riemannian foliation on a complete Riemannian manifold $M$. If $\mathcal{F}$ is closed, we call $\f$ \emph{taut} if every leaf of $\f$ is taut. 
\end{dfn}

Note that if $\f$ is closed, then by Lemma \ref{lem:prop} all the leaves are embedded submanifolds. Further, if $\f$ is the trivial foliation given by the points of $M$ and $\f$ is taut, then, as already defined in Section \ref{sec:LC}, we call $M$ \emph{pointwise taut}.\\

In order to prove tautness of a foliation, one can always assume that $M$ is simply connected. To see this, we first need the closeness property of lifts.

\begin{lem}
\label{tlem:1}
Let $\mathcal{F}$ be closed and let $\pi : N \To M$ be a covering map. Then the lift $\widetilde{\mathcal{F}}$ of $\mathcal{F}$ given by the involutive singular distribution $T\widetilde{\mathcal{F}} = \pi^*(T\mathcal{F})$ is closed.
\end{lem}

\begin{rem} Note that the converse is false, as one can see by the dense torus foliation induced by the submersion $f : \R^2 \To \R, (x,y) \MTo y - \lambda x$, where $\lambda$ is irrational.
\end{rem}   

\begin{proof}
For every leaf $L$ of $\f$, the preimage $\pi^{-1}(L) = \bigcup_i \widetilde{L}_i$ is a union of leaves of $\widetilde{\f}$ and the restriction $\pi|_{\widetilde{L}_i} : \widetilde{L}_i \To L$ is a covering projection for each $i$. Thus each leaf $\widetilde{L} \in \widetilde{\f}$ is a connected component of the closed saturated set $\pi^{-1}(\pi(\widetilde{L}))$ and hence closed. 
\end{proof}

Assume that $f: N \To M$ is a Riemannian submersion between complete Riemannian manifolds and $L \subset M$ is a closed submanifold. Then, by \cite{H}, the map $f: N \To M$ is a locally trivial fiber bundle and therefore, for any point $\bar q \in f^{-1}(q)$, the spaces $\mathcal P(N,f^{-1}(L) \times \bar{q})$ and $\mathcal P(M,L \times q)$ are homotopy equivalent. Since $f$ yields a 1:1 correspondece between the critical points and preserves their indices (cf. Lemma 6.1 in \cite{HLO}), we obtain

\begin{lem}
\label{tlem:1b}
If $f: N \To M$ is a Riemannian submersion between complete Riemannian manifolds and $L \subset M$ is a closed submanifold, then $L$ is taut if and only if $f^{-1}(L)$ is taut.
\end{lem} 

Now, since the homology of a path connected component injects in the homology of the whole space, it is not hard to see that a union of connected, closed submanifolds is taut if and only if its components are taut. So that we deduce   

\begin{lem}
\label{tlem:2}
Let $\pi: N \To M$ be a Riemannian covering and let $M$ be complete. If $\mathcal{F}$ is closed, then $\mathcal{F}$ is $\F$-taut if and only if the lift $\widetilde{\mathcal{F}}$ of $\mathcal{F}$ to $N$ is $\F$-taut. 
\end{lem}  

Given a closed singular Riemannian foliation $\f$ on a complete Riemannian manifold $M$, every leaf posses a global $\varepsilon$-tube. For a regular leaf $L$ with such a global tube, the restriction of the foot point projection on a nearby regular leaf induces a finite covering map onto $L$. We say that $L$ has \emph{trivial holonomy} if all these coverings are diffeomorphisms. It is well known that the set of regular points whose leaves have trivial holonomy is open and dense in $M$. 

In particular, all regular leaves of $\f$ have trivial holonomy, that is to say that the quotient $M_0/\f$ is a Riemannian manifold, if the foliation is taut and $M$ is simply connected. To see this, for $p,q \in M$, let $\Omega_{p,q}(M)$ denote the space of all paths from $p$ to $q$. Then $\Omega_{p,q}(M) \simeq \Omega_{q,q}(M)$ and the long exact sequence of the path space fibration gives $\pi_i(M) \cong \pi_{i-1}(\Omega_{q,q}(M))$, what implies that $\Omega_{p,q}(M)$ is connected.

The fibration $\mathcal P(M,L \times q) \To L$ given by $c \MTo c(0)$ gives the below part of the corresponding long exact homotopy sequence
\begin{equation*}
\pi_0(\Omega_{q,q}(M)) \To \pi_0(\mathcal{P}(M,L \times q)) \To \pi_0(L) \To 1.
\end{equation*}
\noindent Thus $\mathcal{P}(M,L \times q)$ is connected. Now a leaf with nontrivial holonomy would yield at least two local minima for the energy on the path space of a neighboring generic leaf, i.e. a leaf without holonomy. By tautness, all the maps in homology are injective what clearly contradicts our connectedness observation.

We are now able to state a characterisation of taut regular foliations, that indeed also follows from our second main result, Theorem \ref{thm:B} below. For the notion of Riemannian orbifolds see Section \ref{sec:SP}. 
 
\begin{lem}
\label{thm:main}
Let $\f$ be a closed (regular) Riemannian foliation on a complete Riemannian manifold $M$. Then $\f$ is $\F$-taut if and only if the quotient $M/\f$ is a good Riemannian orbifold with a pointwise $\F$-taut universal covering orbifold, i.e. $M/\f$ is isometric to $N/\Gamma$ with a simply connected Riemannian manifold $N$ all of whose points are $\F-$taut and $\Gamma \subset \mr{Iso}(N)$ is a discrete subgroup of isometries.
\end{lem}

\begin{proof}
Let $\widetilde{\f}$ denote the lift of $\f$ to the universal cover $\widetilde{M}$ of $M$. Then, by Lemma \ref{tlem:2}, $\widetilde{\f}$ is taut if and only if $\f$ is taut. By the discussion above, if $\widetilde{\f}$ is taut, then it is simple, i.e. given by the fibers of a Riemannian submersion. So if $\f$ is taut, the quotient map $\widetilde{M} \To \widetilde{M}/\widetilde{\f}$ is a Riemannian submersion between complete Riemannian manifolds and $\widetilde{M}/\widetilde{\f}$ is $1$-connected by the exact sequence for fibrations. In this case, the map $\widetilde M/\widetilde{\f} \To M/\f$ coincide with the universal orbifold covering. On the other hand, assume that $M/\f$ is a good Riemannian orbifold. Then, due to Lemma \ref{lem:orbicov}, $\widetilde{\f}$ is simple, that is to say $\widetilde{M}/\widetilde{\f}$ is a Riemannian manifold and we can reduce the problem to Lemma \ref{tlem:1b}. 

\end{proof} 

\begin{rem}
At the end of Section \ref{sec:PQ} we prove the corresponding statement for the class of foliations whose quotients are orbifolds and coefficients in $\Z_2$.
\end{rem}

We end this section with some genericity results.

\begin{lem}
\label{tlem:3}
Let $\f$ be closed. Then $L \in \f$ is $\F$-taut if and only if the energy functional $E_q : \mathcal{P}(M, L \times q) \To \R$ is an $\F$-perfect Morse function for all non $L$-focal regular points $q \in M$. Further, $\f$ is $\F$-taut if and only if all regular leaves are $\F$-taut. 
\end{lem}

\begin{proof}
Since the set of non-focal points of $L$ as well as the set of regular points is open and dense in $M$, every neighborhood of a given point $q$ contains a regular point that is not a focal point. Therefore, the same argument as in the proof of Proposition $2.7$ in \cite{TT} yields the first claim.

For the second claim, assume that all regular leaves are $\mathbb F$-taut. Let $N$ be a singular leaf and let $q \in M$ be not a focal point of $N$. By our above observations, we can assume that the point $q$ is regular. Let $\gamma$ be a critical point of $E_q^{N}$. Then $\dot{\gamma}(0)$ is a regular vector of $\mf F_{\gamma(0)}$, hence there exists an $\varepsilon > 0$ such that $L = L_{\gamma(\varepsilon)}$ is a regular leaf contained in a global tube of $N$ and the point $q$ is not a focal point of $L$. Denote by $\bar{\gamma}$ the restriction $\gamma |_{[\varepsilon,1]}$ after linear reparameterization on $[0,1]$. Then the horizontal geodesic $\bar{\gamma}$ is a critical point of the perfect Morse function $E_q^{L} : \mathcal{P}(M, L \times q) \To \R$. Now let $i = \mr{ind}(\bar{\gamma})$ be the index of the geodesic $\bar{\gamma}$ and denote by $\kappa = E_q^{L}(\bar{\gamma})$ its energy. By notational reasons let us set $\mathcal{L}^c = (E_q^L)^{-1}([0,c])$, resp. $\mathcal{N}^c = (E_q^N)^{-1}([0,c])$. Denote by $\sigma \in H_i(\mathcal{L}^{\kappa})$ the completion of the local unstable manifold representing a nontrivial cycle in $H_i(\mathcal{L}^{\kappa},\mathcal{L}^{\kappa}\setminus \left\{\bar{\gamma}\right\})$ associated to $\bar{\gamma}$. We then have $\mr{ind}(\gamma) = \mr{ind}(\bar{\gamma})$ for small numbers $\varepsilon$, by continuity reasons. 

The restriction of the foot point projection $R : L \To N, (p,v) \MTo p$, induces a map  
 \[
 \begin{array}{cccl}
    \bar{R} : 
    & \mathcal{L}^{\kappa} 
    & \To
    & \mathcal{N}^{E_q^N(\gamma)}
    \\
    & c 
    & \MTo
    & \tilde{c}
 \end{array}
 \]
 where $\tilde{c}$ is the curve that one gets by concatenation of the unique horizontal geodesic from $R(c(0))$ to $c(0)$ with $c$ followed by reparameterization between $0$ and $1$ and the map $\bar{R}$ maps level sets to level sets. Moreover, $\bar{R}$ is clearly an immersion. Therefore, $\bar{R}_*(\sigma)$ is a cycle in $\mathcal{N}^{E_q^N(\gamma)}$ that can be deformed within a morse chart around $\gamma$ into a cycle $z$ that agrees with the unstable manifold at $\gamma$ above the $E_q^N$-level $E_q^N(\gamma) - \delta$ for small $\delta$. It follows that the homology class of $z$ and thus the homology class $\bar{R}_*(\sigma)$ is mapped onto a generator of $H_i(\mathcal{N}^{E_q^N(\gamma)},\mathcal{N}^{E_q^N(\gamma)}\setminus \left\{\gamma\right\})$. Since the critical point $\gamma$ was chosen arbitrary, every local unstable manifold can be completed to a cycle in $H_{i(\gamma)}(\mathcal{N}^{E_q^N(\gamma)+ \delta})$, i.e. the map $H_n(\mathcal{N}^{\lambda+ \delta}) \To H_n(\mathcal{N}^{\lambda + \delta},\mathcal{N}^{\lambda- \delta})$ is surjective for all $n$ and regular values $\lambda \pm \delta$. Hence $N$ is $\mathbb F$-taut.   
\end{proof}

Our last observation in this section is that a foliation $\f$ is $\F$-taut if and only if a dense family of regular leaves is $\F$-taut, where we call a family of leaves \emph{dense} if their union is a dense set. The next lemma shows that tautness is a closed property relative to non-collapsing convergence. It is then straight forward to see that tautness of a dense family of regular leaves forces a foliation to be taut. 

\begin{lem}
\label{tlem:5}
Let $\f$ be a closed singular Riemannian foliation on a complete manifold $M$ and let $\left\{ L_n \right\}$ be a sequence of $\F$-taut regular leaves converging to a regular leaf $L$ without holonomy, i.e. for every tubular neighborhood $T$ of $L$ there is a number $n_0 \in \N$ such that $L_n \subset T$ and the canonical projection $\pi : T \To L$ restricted to $L_n$ is a diffeomorphism for every $n \geq n_0$. Then $L$ is $\F$-taut.
\end{lem}

\begin{proof}
Let $T$ be a tubular neighborhood of $L$ and let $\pi : T \To L$ be the canonical projection. Choose a number $n_0 \in \N$ so large that $L_n \subset T$ for all $n \geq n_0$. Now let $q \in M$ be not a focal point of $L$. Then for large $n$, the point $q$ is not a focal point of $L_n$ as well. By $f_n : \mathcal{P}(M,L \times q) \To \mathcal{P}(M,L_n \times q)$, $g_n : \mathcal{P}(M,L_n \times q) \To \mathcal{P}(M,L \times q)$ respectively, we denote the induced maps between the path spaces which one gets by assigning to a curve $c$ the curve $\gamma_{c(0)}\cdot c$ and then reparameterizing it between $0$ and $1$, where $\gamma_{c(0)}$ is the unique shortest geodesic between $L_n$ and $L$ that intersects $L$ in $c(0)$, resp. $\gamma_{c(0)}$ is the unique shortest geodesic between $L$ and $L_n$ that intersects $L_n$ in $c(0)$. Then $f_n$ is in an obvious way a homotopy equivalence with homotopy inverse $g_n$.

Let $\gamma$ be a critical point of $E_q^L$ with $\kappa = E_q^L(\gamma)$. We can choose $n$ so large, i.e. a tube $T$ so small, that there is an $\varepsilon >0$ such that $(\kappa-3\varepsilon,\kappa+3\varepsilon)\setminus \{\kappa\}$ contains only regular values and 
$g \circ f(P^{\kappa-2\varepsilon}) \subset P^{\kappa-\varepsilon}$ with $f = f_n, g = g_n$ and $P^r = \mathcal{P}(M,L \times q)^r$ and $P_n^r$ defined analogous. Moreover, we can deform $g \circ f : P^{\kappa - 2\varepsilon} \To P^{\kappa -\varepsilon}$ into the inclusion $j: P^{\kappa-2\varepsilon} \hookrightarrow P^{\kappa -\varepsilon}$ below the $\kappa$-level of $E_q^L$, i.e. 
\[
  \begin{array}{cccl}
    (g \circ f)_{\ast} = j_{\ast} :
    & H_{\ast}(P^{\kappa - 2\varepsilon})
    & \To 
    & H_{\ast}(P^{\kappa - \varepsilon}).
  \end{array}
\]
Since $P^{\kappa - 2\varepsilon}$ is a strong deformation retract of $P^{\kappa - \delta}$ for all $2\varepsilon > \delta > 0$, the above map is an isomorphism. In particular, the induced map in homology 
\begin{equation*}
f_{\ast} :H_{\ast}(P^{\kappa - 2\varepsilon}) \To H_{\ast}(P_n^{\alpha})
\end{equation*}
\noindent with $\alpha = \max \{E_q^{L_n}(f(c)) : c \in (E_q^L)^{-1}(\kappa - 2\varepsilon) \}$ is injective.
   
Denote by $i = \mr{ind}(\gamma)$ the index and set $\tilde{\alpha} = \max \{E_q^{L_n}(f(c)) : c \in (E_q^L)^{-1}(\kappa) \}$.  Now consider the commutative diagramm which comes from the long exact sequence for pairs of spaces together with the natural behavior of the connecting homomorphism
 \[
 \begin{xy}
   \xymatrix
   { 
   H_i(P^{\kappa}) \ar[r]^{f_{\ast}} \ar[d]
   & H_i(P_n^{\tilde{\alpha}}) \ar[d]
   \\
   H_i(P^{\kappa},P^{\kappa- 2\varepsilon}) \ar[r]^{f_{\ast}} \ar[d]_{\partial_{\ast}} 
   & H_i(P_n^{\tilde{\alpha}},P_n^{\alpha}) \ar[d]^{\tilde{\partial}_{\ast}}
   \\
   H_{i-1}(P^{\kappa -2\varepsilon}) \ar[r]^{f_{\ast}}
   & H_{i-1}(P_n^{\alpha})
   \\
   }
 \end{xy}
\]
Since $L_n$ is taut, we have $\tilde{\partial}_{\ast} = 0$. So 
\begin{equation*}
f_{\ast} \circ \partial_{\ast} = \tilde{\partial}_{\ast} \circ f_{\ast} = 0.
\end{equation*}
As we have seen, the map $f_{\ast} : H_{\ast}(P^{\kappa - 2\varepsilon}) \To H_{\ast}(P_n^{\alpha})$ is injective. But this means $\partial_{\ast} = 0$, i.e. $L$ is taut.            
\end{proof}

Now assume under the assumptions of Lemma \ref{tlem:5} that all $L_n$ are regular leaves without holonomy and that $L$ is an exceptional leaf, i.e. has nontrivial holonomy. Due to Lemma \ref{tlem:2}, we can assume that $M$ is simply connected. Then for large $n$, the leaf $L$ would provide at least two local minima for $L_n$. Again, by tautness of $L_n$, the path space corresponding to $L_n$ would be disconnected. But this is clearly a contradiction, since $M$ is simply connected. So, $L$ necessarily has trivial holonomy. Thus combining Lemma \ref{tlem:2}, Lemma \ref{tlem:3}, and Lemma \ref{tlem:5} together with the fact that the set of regular leaves without holonomy is open and dense in $M/\f$, we have

\begin{cor}
\label{tlem:7}
The closed singular Riemannian foliation $\f$ is $\F$-taut if and only if a dense family of leaves is $\F$-taut.
\end{cor}

\subsection{Index Splitting for Horizontal Geodesics}
\label{sec:HGI}

In this section we summarize some observations on the focal indices of horizontal geodesics in a singular Riemannian foliation $(M,\f)$. We will see that the focal data of the space of normal $L_{\gamma(a)}$-Jacobi fields along a regular horizontal geodesic $\gamma: [a,b] \To M$ are of two types. Namely, for $t \in (a,b)$, there is a vertical multiplicity $\mr{dim}(\f) - \mr{dim}(L_{\gamma(t)})$ that counts the intersections with the singular stratum, and a horizontal multiplicity that is, roughly speaking, the multiplicity of $\gamma(t)$ as a conjugate point of $\gamma(a)$ in the quotient $M/\f$ along the projection of $\gamma$. We refer to \cite{L07}, \cite{LT} and \cite{Wi} for the proofs and a more detailed discussion of the following facts. Our summary is referred to \cite{LT}.\\   

If $\f$ is a singular Riemannian foliation on a Riemannian manifold $(M,g)$ we call a geodesic $\gamma$ \emph{horizontal} if it meets one leaf  of $\f$, and hence all the leaves it meets, perpendicularly. We will call such a geodesic $\gamma : [a,b] \To M$ regular, if the endpoints lie on regular leaves of $\f$.

A regular horizontal geodesic intersects the singular set $M\setminus M_0$ of $\f$ only in finitely many points $a < t_1 < \dots < t_r < b$ (see Cor.4.6 in \cite{LT}) and the number
\begin{equation*}
c(\gamma) = \sum_{i = 1}^r\mr{dim}(L_{\gamma(a)}) - \mr{dim}(L_{\gamma(t_i)})
\end{equation*}
is  called the \emph{crossing number} of $\gamma$. 

\begin{dfn}
Let $\gamma : [a,b] \To M$ be a horizontal geodesic. A Jacobi field $J$ along $\gamma$ is called an $\f$-\emph{Jacobi field} if it is the variational vector field of a variation through horizontal geodesics starting on the same leaf as $\gamma$.  It is called an $\f$-\emph{vertical Jacobi field} if it satisfies $J(t) \in T_{\gamma(t)}L_{\gamma(t)}$ for all $t$.
\end{dfn}

For a geodesic $\gamma : [a,b] \To M$ the space $\mr{Jac}$ of all normal Jacobi fields along $\gamma$ is a symplectic vector space with its canonical symplectic form 
\begin{equation*}
\omega(J_1,J_2) = \left <\nabla J_1,J_2 \right > + \left <J_1,\nabla J_2 \right >.
\end{equation*}
 Now let $(M,\f,g)$ be a singular Riemannian foliation and let $\gamma : [a,b] \To M$ be a horizontal geodesic. To arrange the dependence of $\Lambda^{L_{\gamma(a)}}$ on the starting point of the geodesic consider the isotropic space $W^{\gamma}$ consisting of all Jacobi fields along $\gamma$ with the property that these fields are variational fields through horizontal geodesics that intersect the same leaves simultaneously, i.e. of variations $\gamma_s$ with $\gamma_s(t) \in L_{\gamma(t)}$ for all $t$. It is shown in  \cite{LT}  that $W^{\gamma}(t) = \left\{J(t) \mbox{ $|$ } J \in W^{\gamma} \right\}$ coincides with $T_{\gamma(t)}L_{\gamma(t)}$, for all $t$ and by definition we have $W^{\gamma} \subset \Lambda^{L_{\gamma(a)}}$. The space $W^{\gamma}$ is exactly the space of all $\f$-vertical Jacobi fields along $\gamma$ and does not depend on the starting point. We have $\mr{dim}(W^{\gamma}) = \max_{t \in [a,b]}\{\mr{dim}(L_{\gamma(t)})\}$ and we denote this number by $d(\gamma)$. The $W^{\gamma}$-focal points along $\gamma$ are precisely the points $\gamma(t)$, where $\gamma$ crosses the singular strata, i.e. when $\mr{dim}(L_{\gamma(t)}) < d(\gamma)$, and the $W^{\gamma}$-focal index of $\gamma(t)$ 
 \begin{equation*}
 \mr{dim}(\{J \in W^{\gamma}\mbox{$|$} J(t) = 0\})
 \end{equation*}
  is equal to $d(\gamma) - \mr{dim}(L_{\gamma(t)})$. In particular, for a regular horizontal geodesic $\gamma$ its crossing number $c(\gamma)$ coincides with the index $\mr{ind}_{W^{\gamma}}(\gamma)$.

For notational convenience set $L = L_{\gamma(a)}$ and $W = W^{\gamma}$. 
It follows from the observation of Wilking in \cite{Wi}  that the extension $\widetilde W$ of $W$, given by
\begin{equation*}
\widetilde{W}(t) = W(t) \oplus \left\{ \nabla J(t) \mbox{ $|$ } J \in W, J(t) = 0 \right\},
\end{equation*}
defines a smooth subbundle of the normal bundle $N$ of $\gamma$. Of course, $\widetilde W$ coincides with $W$ along the regular part of $\gamma$. Using the identification of $N/\widetilde W$ with the orthogonal complement $H$ of $\widetilde W$ in $N$, there is a smooth map $A : \widetilde{W} \To H$, defined by $A(J(t)) = P(\nabla J(t))$ and $A(\nabla J(t)) = 0$ if $J(t) = 0$, where $P$ denotes the orthogonal projection. Let $R : N \To N$ denote the curvature endomorphism, given by $R(X) = R(X,\dot{\gamma})\dot{\gamma}$ and consider the induced family of symmetric endomorphisms $R^H : H \To H$, defined by 
 \begin{equation*}
 R^H(Y) = P(R(Y)) + 3AA^*(Y).
 \end{equation*}
Then the (more general) work of Wilking on  the transversal Jacobi equation in \cite{Wi} shows that the $R^H$-Jacobi fields, i.e. solutions of $(\nabla^H)^2 J + R^H(J) = 0$, where $\nabla^H(J) = P(\nabla J)$ denotes the induced covariant derivative on $H$, are precisely the projections of the Jacobi fields in the symplectic orthogonal complement of $W$. In particular, since the space $\Lambda^L$ is Lagrangian and contains $W$, one obtains the next lemma (cf. \cite{L07} and \cite{LT}).

\begin{lem}
\label{lem:indexsplitting}
If, for an isotropic  subspace $V \subset \mr{Jac}$, we denote by 
\begin{equation*}
\mr{ind}_V(\gamma) = \sum_{t \in [a,b]}(\mr{dim}(V) - \mr{dim}(\{J \in V \mbox{$|$} J(t) = 0\}))
\end{equation*}
 the $V$-focal index of $\gamma$, we have the equality 
 \begin{equation*}
 \mr{ind}_W(\gamma) + \mr{ind}_{\Lambda^L/W}(\gamma) = \mr{ind}_{\Lambda^L}(\gamma).
 \end{equation*}
  Further, the  index  $\mr{ind}_{\Lambda^L/W}(\gamma)$ is an intrinsic datum of the quotient $M/\f$.  
\end{lem}

By the latter statement we mean the following. Assume that two singular Riemannian foliations $(M_i,\f_i), i=1,2$, have isometric  quotients and identify $M_1/\f_1 = M_2/\f_2$ via an isometry. Then for any two horizontal geodesics which start in the respective regular part and have the same projection to the quotient, their horizontal indices coincide. This follows from the fact that their transversal Jacobi equations coincide outside a set of isolated points. Of course, this observation is valid if one only requires neighborhoods of the projections in the respective quotients to be isometric.

\begin{dfn}
For a horizontal geodesic $\gamma$, we call $\mr{ind}_{\Lambda^L/W}(\gamma)$ the \emph{horizontal index of  $\gamma$} and $\mr{ind}_W(\gamma) $ the \emph{vertical index of $\gamma$}.
\end{dfn}

\begin{exl}
In the case of a Riemannian submersion $f: M \to M/\f$, the horizontal geodesics are the horizontal lifts of geodesics in $M/\f$ and the space $W^{\gamma}$ of such a horizontal geodesic is exactly the space of variational vector fields corresponding to variations through horizontal lifts of $f \circ \gamma$. In this case, the vertical index is zero, as it counts the intersections with singular leaves, and the term $AA^*$ in the above equation is just the  O'Neill tensor, so that the indices of $\gamma$ and $f \circ \gamma$ coincide.
\end{exl}

\subsection{A Property of the Quotient}
\label{sec:PQ}

Dealing with singular Riemannian foliations one focuses mainly on the horizontal geometry of the foliation, that is to say the geometry of the quotient. For this reason, one is often interested in geometric properties of the foliation that can be read off the quotient and to consider \emph{equivalence classes} of foliations by means of isometric quotients. An example of such a quotient property is infinitesimal polarity (cf. \cite{LT}) what is equivalent to the property that the quotients are Riemannian orbifolds. Our second main result now states that tautness of a foliation is actually also a property of the quotient, so that one can speak about \emph{equivalence classes} of taut foliations by means of their leaf spaces.

\begin{thm}
\label{thm:B}
Let $\f$ and $\f'$ be closed singular Riemannian foliations on complete Riemannian manifolds $M$ and $M'$ with isometric quotients. Then $\f$ is taut if and only $\f'$ is taut. In particular, if one of them is $\mathbb F$-taut, then both are $\mathbb Z_2$-taut.
\end{thm}

Before we begin with the proof of the theorem let us discuss and apply this result in the context of the known examples.\\

If $M = S^k$ is the round sphere and $\f$ is the trivial foliation by points, there is a well known cycle construction for critical points of the energy functional (cf. p.95-96 of \cite{M}) which shows that $S^k$ is pointwise taut. Terng and Thorbergsson proved in \cite{TT} that the standard metric on the sphere is the only one with respect to which the sphere is pointwise taut. 

Now consider the more general case $M/\f = N/\Gamma$, where $N$ is a symmetric space.
In their study of Morse theory of symmetric spaces, Bott and Samelson came up with concrete cycles which represent a basis in $\Z_2$-homology of generic path spaces $\mathcal P(N,p \times q)$ and which are in fact compact connected manifolds (see \cite{BS}) and coincide with those cycles we constructed in Theorem \ref{thm:A}. In particular, symmetric spaces are pointwise taut . Therefore, the foliation $\f$ on $M$ has to be taut by Theorem \ref{thm:B}. Of course, one could allow an additional constant direction of non-positive curvature in the quotient, because under this assumption there are no focal points in this direction.

We want to emphasize that the following corollary covers all known examples.

\begin{cor}
If $\f$ is a closed singular Riemannian foliation on a complete Riemannian manifold $M$ and $M/\f = (N \times P)/\Gamma$ is a good Riemannian orbifold, where $N$ is a symmetric space and $P$ is a non-positively curved manifold, then $\f$ is taut.
\end{cor}

Another application of Theorem \ref{thm:B} are foliations admitting generalized sections.

\begin{exl}
Let $M$ be a complete Riemannian manifold with an isometric action of a compact Lie group $G$. In \cite{GOT} the authors developed the concept of a generalized section for such an action. They call a connected, complete submanifold $\Sigma$ of $M$ a $k$-\emph{section} if the following hold:
\begin{itemize}
\item $\Sigma$ is totally geodesic;
\item $\Sigma$ intersects all orbits;
\item for every $G$-regular point $p \in \Sigma$ the tangent space $T_p\Sigma$ contains the normal space $\nu_p(G(p))$ as a subspace of codimension $k$;
\item if $p \in \Sigma$ is a $G$-regular point with $g(p) \in \Sigma$ for some $g \in G$ then $g(\Sigma) = \Sigma$.
\end{itemize}
Generalized sections are also called \emph{fat sections} and the copolarity of $(G,M)$ is defined by $\mr{copol}(G,M)=\min \left\{k \in \N \mbox{$|$} \text{ there is a $k$-section } \Sigma \subset M\right\}$ and measures, roughly speaking, how far the action is from being polar, i.e. admitting a $0$-section. If $\Sigma$ is a fat section, then it is shown in \cite{Ma} that there is the \emph{fat Weyl group} $W(\Sigma) = N_G(\Sigma)/Z_G(\Sigma)$ that acts on $\Sigma$ with $G(p) \cap \Sigma = W(\Sigma)(p)$ if $p \in \Sigma$, inducing an isometry between the quotients $\Sigma/W(\Sigma) = M/G$. We therefore deduce that $(\Sigma,\f^W)$ is taut if and only if $(M,\f^G)$ is taut.
\end{exl} 

Before we start with the proof of Theorem \ref{thm:B}, we now state a preparing lemma that says that focal points caused by singular leaves do not provide any difficulties when dealing with tautness. This fact was already discussed in \cite{EN}.

\begin{lem}
\label{lem:fiber}
Let $\f$ be a closed singular Riemannian foliation on a complete Riemannian manifold $M$ and let $L \in \f$ be a regular leaf. For every broken horizontal geodesic $c : [0,1] \To M$ from $L$ to a point $q \in M$ that intersects the singular stratum discretely, let $\Delta(c)$ denote the space of broken horizontal geodesics in the path space $\mathcal P(M,L\times q)$ which have the same projection to the quotient $M/\f$ as $c$. Then $\Delta(c)$ carries a smooth structure of a compact (possibly non-connected) manifold of dimension $\sum_{t \in [0,1]}\mr{dim}(\f) - \mr{dim}(L_{c(t)})$ such that the inclusion into the path space $\mathcal P(M,L \times q)$ becomes an embedding. 
\end{lem}

\begin{proof}
Given a leaf $L \in \f$ let $\nu^{\varepsilon}(L)$ be a global $\varepsilon$-tube of $L$. Then the pull back of $\f$ by the normal exponential map is invariant under the homotheties $r_{\lambda}(v) = \lambda v$ for all $\lambda \in [-1,1]\setminus \{ 0 \}$, so that there is a unique singular foliation $\mathcal G(L)$ that extends the pull back to $\nu(L)$ satisfying this property. The singular foliation $\mathcal G(L)$ is closed if $\f$ is closed and it is shown in Section 4 of \cite{LT} that $v,w \in \nu(L)$ are in the same leaf of $\mathcal G(L)$ iff $\gamma_w(t) \in L_{\gamma_v(t)}$ for all $t$, where as usual $\gamma_v$ is the unique geodesic with $\dot{\gamma}_v(0) = v$. Let $V$ be a small open neighborhood of $v$ in the leaf $\mathcal L_v$ of $\mathcal G(L)$ through $v$. Then the vector space of variational vector fields of variations through geodesics $\gamma_w$ with $w \in V$ coincides with the space $W^{\gamma_v}$ of $\f$-vertical Jacobi fields along $\gamma_v$, as defined in the last section. Due to \cite{LT}, one has
\begin{equation*}
W^{\gamma_v}(t) = \left\{J(t) \mbox{$|$} J \in W^{\gamma_v}\right\} = T_{\gamma_v(t)}L_{\gamma_v(t)},
\end{equation*}
and we deduce that the map $\eta_t: \mathcal L_v \To L_{\gamma_v(t)}$, given by $\eta_t(w) = \exp(tw) = \gamma_w(t)$, is a submersion for all $t$, which is surjective if $\f$ is closed. In this case, all the preimages $\eta_t^{-1}(p)$ are compact submanifolds of $\mathcal L_v$ of dimension $\mr{dim}(\mathcal L_v) - \mr{dim}(L_{\gamma_v(t)})$. In particular, if $L$ is a regular leaf, the dimension of such a preimage equals the difference $\mr{dim}(\f) - \mr{dim}(L_{\gamma_v(t)})$.\\

We will now describe the compact set $\Delta(c)$ as the total space of an iterated bundle. Since the general case requires no new ideas, but only some more notation, we will assume for the rest of the proof that $c$ as in the claim is smooth. So let $L \in \f$ be a regular leaf and let $\gamma = \gamma_v$ be a horizontal geodesic from $L$ to a point $q \in M$. Let $\gamma^{-1}(M \setminus M_0) = \{ t_i \}_{i = 1,\dots,r}$ with $0 < t_r < \dots < t_1 \leq 1$ denote the times where $\gamma$ crosses the singular stratum and set $L_i = L_{\gamma(t_i)}$ and $v_i = \mr{dim}(\f) - \mr{dim}(L_i)$. Note that if $q$ is a regular point, the vertical index of $\gamma$ is given by $v(\gamma) = \sum_{i=1}^r v_i$. With the notation from above, let $\eta_i: \mathcal L_v \To L_i$ be the surjective submersion defined by $\eta_i(w) = \exp(t_iw)$. Starting with the furthermost singular leaf, we now define $V_1 = \eta_1^{-1}(\gamma(t_1)) \subset \mathcal L_v$ and identify this space with the subspace  
\begin{equation*}
\Delta_1 = \left\{c_w \in \Delta(\gamma) \mbox{$|$} c_w|_{[0,t_1]} = \gamma_w|_{[0,t_1]} \text{ for } w \in V_1 \text{ and } c_w|_{[t_1,1]} = \gamma |_{[t_1,1]} \right\}
\end{equation*}
of $\Delta(\gamma)$ of (at most) once broken geodesics in the obvious way, i.e. by $w \MTo c_w$. With this identification $\Delta_1$ inherits a smooth structure which turns it into an embedded submanifold of $\mathcal P(M,L \times q)$ of dimension $v_1$.\\

At the second step we define $V_2$ to be the twisted product
\begin{equation*}
\mathcal L_v \times_{\eta}V_1 = \left\{(w_2,w_1) \in \mathcal L_v \times V_1 \mbox{$|$} \eta_2(w_2) = \exp(t_2w_1)\right\},
\end{equation*}  
which can be identified with the subspace $\Delta_2$ of $\Delta(\gamma)$ that consists of all (at most) twice broken horizontal geodesics $c_{(w_2,w_1)}$ with $c_{(w_2,w_1)}|_{[0,t_2]} = \gamma_{w_2}|_{[0,t_2]}$ for some element $w_2 \in \mathcal L_v$ and $c_{(w_2,w_1)}|_{[t_2,1]} = c_{w_1}|_{[t_2,1]}$ for some $w_1 \in V_1$. With the induced smooth structure $\Delta_2$ becomes a submanifold of $\mathcal P(M,L \times q)$ with 
\begin{equation*}
\mr{dim}(\Delta_2) = \mr{dim}(\mathcal L_v) + \mr{dim}(V_1) - \mr{dim}(L_2) = v_2 + v_1.
\end{equation*}
Note that all we need to ensure that $\Delta_2$ is a submanifold is the fact that the map $\eta_2: \mathcal L_v \To L_2$ is a submersion, so that $\mathcal L_v \times V_1 \To L_2 \times L_2$ is transversal to the diagonal in $L_2 \times L_2$.\\  

Now assume that for some $r-1 \geq j \geq 1$ we already have defined $V_j$ as a submanifold of dimension $\sum^j_{i = 1}v_i$ of the $j$-fold product $\mathcal L_v^j$ together with an identification $V_j \cong \Delta_j$ given by $(w_j,\dots,w_1) \MTo c_{(w_j,\dots,w_1)}$. Then we inductively define $V_{j+1}$ and $\Delta_{j+1}$ as follows. Set $V_{j+1} = \mathcal L_v \times_{\eta}V_j$, where again the twisted product is defined by
\begin{equation*}
\mathcal L_v \times_{\eta}V_j = \left\{(w_{j+1},w_j,\dots,w_1) \in \mathcal L_v \times V_j \mbox{$|$} \eta_{j+1}(w_{j+1}) = \exp(t_{j+1}w_j)\right\},
\end{equation*} 
which is therefore a submanifold of $\mathcal L_v^{j+1}$ of dimension 
\begin{eqnarray*}
\mr{dim}(V_{j+1}) & = & \mr{dim}(\mathcal L_v) + \mr{dim}(V_j) - \mr{dim}(L_{j+1})\\
                  & = & v_{j+1} + \mr{dim}(V_j)\\
                  & = & \sum_{i = 1}^{j+1}v_i .
\end{eqnarray*}

Finally, define $\Delta_{j+1}$ to be the subspace of $\Delta(\gamma)$ consisting of all (at most) $(j+1)$-fold broken horizontal geodesics $c_{(w_{j+1},w_j,\dots,w_1)}$ such that 
\begin{eqnarray*}
& c_{(w_{j+1},w_j,\dots,w_1)}|_{[0,t_{j+1}]} & =  \gamma_{w_{j+1}}|_{[0,t_{j+1}]} \text{ for some } w_{j+1} \in \mathcal L_v \text{ and }\\
& c_{(w_{j+1},w_j,\dots,w_1)}|_{[t_{j+1},1]} & = c_{(w_j,\dots,w_1)}|_{[t_{j+1},1]} \text{ for some } (w_j,\dots,w_1) \in V_j.
\end{eqnarray*}

 By construction, it is clear that there is a 1:1 correspondence between $(j+1)$-tupels $(w_{j+1},\dots,w_1) \in V_{j+1}$ and paths $c_{(w_{j+1},w_j,\dots,w_1)} \in \Delta_{j+1}$. Moreover, the identification $\Delta_{j+1} \cong V_{j+1}$ (as manifolds) via $(w_{j+1},w_j,\dots,w_1) \MTo c_{(w_{j+1},w_j,\dots,w_1)}$ turns $\Delta_{j+1}$ into a compact submanifold of $\mathcal P(M,L \times q)$ of dimension $\sum_{i=1}^{j+1}v_i$. In particular, this defines a smooth structure for $\Delta_r = \Delta(\gamma)$ with the desired properties.  
\end{proof}

\begin{rem}
\label{rem:fiber}
The assumptions in Lemma \ref{lem:fiber} are adapted to our setting, but the conclusion also holds if the foliation is not closed, or the manifold is not complete. For this fact, because being a manifold is a local property, one only has to localize the arguments given in the proof of the lemma. 
Further, if $c = \gamma$ is smooth, let $W^{\gamma}$ denotes the space of $\f$-vertical Jacobi fields along $\gamma$ (see Section \ref{sec:HGI}) and let $t_1 < \dots < t_r$ be the $W^{\gamma}$-focal times along $\gamma$. Define $W^{\gamma}_i$ to be the space of continuous vector fields $J$ along $\gamma$ such that $J|_{[0,t_i]} \in W^{\gamma}|_{[0,t_i]}$ and $J$ vanishes on $[t_i,1]$. Then by our description in the proof of the lemma, we conclude that the tangent space of $\Delta(\gamma)$ at $\gamma$ is given by
\begin{equation*}
T_{\gamma}\Delta(\gamma) = \bigoplus_{i=1}^r W^{\gamma}_i.
\end{equation*}
\end{rem} 

\noindent As a consequence of Lemma \ref{lem:fiber} using the same notation we reprove the mentioned special case.

\begin{cor}
\label{cor:fiber}
If there are no horizontal conjugate points, i.e. conjugate points for the transversal Jacobi equation, along the horizontal geodesic $\gamma$ and $\gamma(1)$ is not a focal point of $L_{\gamma(0)}$, then $\gamma$ is of linking type (for $E_{\gamma(1)}$) with respect to $\Z_2$.
\end{cor}

\begin{proof}
By Lemma \ref{lem:fiber}, there is a compact manifold $\Delta(\gamma)$ through every regular horizontal geodesic $\gamma$ consisting of broken horizontal geodesics, all having the same length as $\gamma$, and $\mr{dim}(\Delta(\gamma))$ coincides with the vertical index $v(\gamma)$. But by assumption, the index of $\gamma$ is just the vertical index. Moreover, the satement about $T_{\gamma}\Delta(\gamma)$ and the discussion at the end of the proof of Theorem \ref{thm:A} ensures that if we look at a finite dimensional approximation of the path space, $\Delta(\gamma)$ is transversal to the ascending cell in a Morse chart around $\gamma$, so that $\Delta(\gamma)$ can be deformed into the descending cell, hence defines a linking cycle for $\gamma$.
\end{proof}

\begin{proof}[Proof of Theorem \ref{thm:B}]
Let us briefly sketch the idea of the proof.
For $(M,\f)$ and $(M',\f')$ as in the claim let us identify $B = M/\f = M'/\f'$ via an isometry and consider the following diagram
\[
  \begin{xy}
    \xymatrix
    {
      M \ar[dr]_Q & & M' \ar[dl]^{Q'}\\
      & B & 
     }
  \end{xy}
  \]
 Now assume that $\f$ is taut. In order to prove that $\f'$ is taut it suffices to prove that the normal exponential map of a generic leaf of $\f'$ has integrable fibers, by Theorem \ref{thm:A} and our genericity results from Section \ref{sec:SRB}. Let therefore $L' \in \f'$ be a regular leaf without holonomy. In this case, the leaf $L = Q^{-1}(Q'(L')) \in \f$ is a regular leaf without holonomy, too, and its normal exponential map has integrable fibers, by assumption. For $v \in \nu(L)$, let $\Delta_v$ denote the connected component of the fiber through $v$ that contains $v$ and identify it with the manifold of horizontal geodesics from $L$ to $\exp(v)$ which have initial velocity in $\Delta_v$. Now, given a vector $v' \in \nu(L')$ with the same projection to $B$ as $v$, we push $\Delta_v$ down to $B$ and lift it to $M'$ along $Q'$ to obtain a space $\Delta'_{v'}$ of horizontal geodesics that start in $L'$ and end in $\exp(v')$. The observation that the map $\Delta'_{v'} \To \nu(L')$ which assigns to a horizontal geodesic its starting direction provides an integral manifold of the kernel distribution of the normal exponential map of $L'$ through $v'$ then finishes the proof.
 
Having sketched the proof, let us now work out the details. Given a point $p \in M$, the infinitesimal foliation $\mf F_p$ splits as a product foliation $\mf F_p = T_pL_p \times \mf F_p^1$ on the tangent space $T_pM = T_pL_p \times \nu_p(L_p)$ so that we have $T_pM/\mf F_p = \nu_p(L_p)/\mf F_p^1$, which is the tangent space to a local quotient $U/\f$ at $L \cap U$, where $U$ is a distinguished neighborhood of $p$. The map $U/\f \To M/\f$, induced by the inclusion $U \To M$, is a finite-to-one open map, given by the quotient map of the action of a finite group $\Gamma$ of isometries, onto a neighborhood $T^{\varepsilon}/\f$ of $L_p$, where $T^{\varepsilon}$ is a global $\varepsilon$-tube around $L_p$ with the same $\varepsilon$ as in the definition of the distinguished neighborhood $U$ (cf. Section \ref{sec:SRB}). Identifying $\nu^{\varepsilon}(L_p) \cong T^{\varepsilon}$ via the normal exponential map, we see that we can identify the tangent space $T_{L_p}B$ of $B$ at $L_p$ with $(\nu_p(L_p)/\mf F_p^1)/\Gamma$. We therefore define the differential $dQ_p : T_pM \To T_{Q(p)}B$ of the projection $Q: M \To B$ at the point $p$ to be the composition of projections $T_pM \To \nu_p(L_p) \To (\nu_p(L_p)/\mf F_p^1)/\Gamma$ and write $Q_*: TM \To TB$ for the induced map, i.e. $Q_*(v) = dQ_{P(v)}(v)$. We use the analogous notations for $Q': M' \To B$.
 
Since orbifold geodesics coincide if they coincide initially, we deduce from the fact that the set of $\f$-horizontal vectors $v$ in $TM$ with the property that $Q \circ \gamma_v$ is completely contained in the open and dense orbifold part of $B$ has full measure in the subset of all horizontal vectors (cf. \cite{LT}) that given an $\f$-horizontal vector $v$ and an $\f'$-horizontal vector $v'$ with the same projection, i.e. $Q_*(v) = Q'_*(v')$, we have $Q \circ \gamma_v(t) = Q' \circ \gamma_{v'}(t)$ for all $t \in \R$. \\

Now take a regular leaf $L \in \f$ and recall that in this case the foliation $\mathcal G(L)$ on $\nu(L)$ from the proof of Lemma \ref{lem:fiber} is a regular foliation with closed leaves such that the intersection of every leaf $\mathcal L \in \mathcal G(L)$ with any normal space $\nu_p(L)$ is finite. Further, as explained there, for every normal vector $v \in \nu(L)$ the restriction of the normal exponential map to the leaf $\mathcal L_v$ induces a submersion $\eta_v: \mathcal L_v \To L_{\gamma_v(1)}$, so that all the fibers $\eta_v^{-1}(q)$ for $q \in L_{\gamma_v(1)}$ are (unions of) compact submanifolds of dimension $\mr{dim}(\f) - \mr{dim}(L_{\gamma_v(1)})$, since the normal exponential map is proper. In this case, the smooth map $\eta_v^{-1}(q) \To \nu_q(L_{\gamma_v(1)})$, given by $w \MTo \dot{\gamma}_w(1)$, defines a smooth identification of the connected component of $\eta_v^{-1}(q)$ that contains $w$ with the regular leaf $\mathfrak L_{\dot{\gamma}_w(1)}$ of $\mf F_q^1$ through $\dot{\gamma}_w(1) \in \nu_q(L_{\gamma_v(1)})$, where $w \in \eta_v^{-1}(q)$ is any preimage of $q$. Moreover, due to the fact that a regular horizontal geodesic intersects the singular stratum discretely, we  see that for $v \in \nu(L)$ and a horizontal geodesic $\gamma'_{v'}:[0,1] \To M'$ with $Q \circ \gamma_v = Q' \circ \gamma'_{v'}$ their horizontal indices coincide, because the corresponding transversal Jacobi equations coincide along the regular parts. That is to say, the kernel $\mr{ker}((d\exp^{\perp}_M)_v)$ of the differential of the normal exponential map in $v$ contains the subspace $T_v\mathcal L_v \subset T_v\nu(L)$ and the dimension of $\mr{ker}((d\exp^{\perp}_M)_v)/T_v\mathcal L_v$ is independent of the foliation, or to be more precise, an intrinsic datum of the quotient.\\

We now finish the proof as follows. Assume that $\f$ is taut. Combining Lemma \ref{tlem:7} and the proof of Theorem \ref{thm:A}, it remains to prove that for generic leaves $L'$ of  $\f'$ the normal exponential map $\exp_{M'}^{\perp}: \nu(L') \To M'$ has integrable fibers. Thus we can restrict our attention to a regular leaf $L' \in \f'$ without holonomy, i.e. $Q'(L')$ is a manifold point of $B$ and the restriction of $Q'$ to a tubular neighborhood of $L'$ defines a Riemannian submersion. In particular, in this case the leaf $L = Q^{-1}(Q'(L')) \in \f$ is a regular leaf without holonomy, too. Let $v' \in \nu(L')$ be a horizontal vector and let us set $q'= \gamma_{v'}(1)$. Choose an $\f$-horizontal vector $v \in \nu(L)$ with $Q_*(v) = Q'_*(v')$ and set $q = \gamma_v(1)$. Then, by construction, $Q \circ \gamma_v = Q' \circ \gamma_{v'}$ and 
\begin{eqnarray*}
&\mr{dim}(\mr{ker}((d\exp^{\perp}_{M'})_{v'})) - (\mr{dim}(\f') - \mr{dim}(L_{\exp^{\perp}_{M'}(v')}))\\ & = \mr{dim}(\mr{ker}((d\exp^{\perp}_{M})_{v})) - (\mr{dim}(\f) - \mr{dim}(L_{\exp^{\perp}_{M}(v)})).
\end{eqnarray*}

Since $\f$ is taut, the connected component $\Delta_v$ of $(\exp^{\perp}_M)^{-1}(q)$ containing $v$ is a compact submanifold of $\nu(L)$ that is smoothly foliated by the $(\mr{dim}(\f)- \mr{dim}(L_q))$-dimensional regular foliation whose leaf through a horizontal vector $w \in \Delta_v$ is given by $\mathcal N_w = (\exp^{\perp}_M)^{-1}(q) \cap \mathcal L_w$. Again, we can regard $\Delta_v$ as a saturated subset of the regular part of the singular Riemannian foliation $\mf F_q^1$ on $\nu_q(L_q)$ via $d_v : \Delta_v \To \nu_q(L_q)$, defined by the prescription $d_v(w) = d(\exp^{\perp}_M)_w(w)$. Moreover, by our choice of $L'$, the image of the composition $dQ_q \circ d_v$ is completely contained in the manifold part of $T_qM/\mathfrak F_q = \nu_q(L_q)/\mf F_q^1$ (cf. Section 4 of \cite{LT}), so that every leaf of $\mf F_{q'}$ through a vector $w' \in \nu_{q'}(L'_{q'})$ with $dQ'_{q'}(w') \in dQ_q(d_v(\Delta_v))$ is also regular without holonomy. 

If we therefore define $K_{q'} \subset \mf F_{q'}^1$ to be the preimage
\begin{equation*}
K_{q'} = (dQ'_{q'})^{-1}((dQ_q \circ d_v)(\Delta_v)),
\end{equation*}
then $K_{q'}$ is obviously a union of regular leaves of $\mf F_{q'}^1$ without holonomy, namely of leaves of dimension $\mr{dim}(\f') - \mr{dim}(L'_{q'})$, completely contained in a concentric sphere. Further, because $dQ_q(d_v(\Delta_v))$ carries a natural smooth structure that turns it into a ($\mr{dim}(\Delta_v) - (\mr{dim}(\f)- \mr{dim}(L_q))$)-dimensional manifold, and the restriction of $dQ'_{q'}$ to the set of points lying on regular leaves without holonomy is a submersion, the set $K_{q'}$ is a compact submanifold of $\nu_{q'}(L'_{q'})$ of dimension equal to $\mr{dim}(\mr{ker}((d\exp^{\perp}_{M'})_{v'}))$.

Now recall that all the infinitesimal foliations are invariant under all non-zero homotheties. Thus, if we define $\Delta'_{v'} = \left\{-\dot{\gamma}_{w'}(1) \in \nu(L')\mbox{$|$} -w' \in K_{v'}\right\}$, it easily follows from our above discussion that $\Delta'_{v'}$ is a compact submanifold of $(\exp^{\perp}_{M'})^{-1}(q')$ containing $v'$ that satisfies $T_{w'}\Delta'_{v'} = \mr{ker}((d\exp^{\perp}_{M'})_{w'})$ for all $w' \in \Delta'_{v'}$. This proves the claim.                 
\end{proof}    

As already mentioned before, tautness of a submanifold $L \subset M$ requires very special symmetry of the pair $(M,L)$ around the submanifold $L$, what clarifies the fact that there are not many examples of taut submanifolds actually known. By this reason, it is worth mentioning that the ideas of the last proof can be used to construct lots of examples. For this purpose, consider a closed singular Riemannian foliation $\f$ on $M$ such that the space of leaves $M/\f$ is a good Riemannian orbifold $N/\Gamma$. Assume that there is a submanifold $S \subset N$ completely contained in the interior of a fundamental domain of the $\Gamma$-action which we identify with $M/\f$, and consider the saturated preimage $T = Q^{-1}(S)$ that is a union of regular leaves without holonomy. Now let $v \in \nu_p(T)$ be a normal vector to $T$. Then every $\f$-vertical Jacobi field along $\gamma_v$ (cf. Section \ref{sec:HGI}) is also a $T$-Jacobi field along $\gamma_v$, i.e. $W^{\gamma_v} \subset \Lambda^T$, and similar arguments as in the proof of Theorem \ref{thm:B} can be used to see that the multiplicity of $Q_*(v)$ as a focal vector of $S$ in $N$ is the same as the difference of the multiplicity of $v$ as a focal vector of $T$ in $M$ and the number $\mr{dim}(\f) - \mr{dim}(L_{\gamma_v(1)})$. Thus the following lemma is obtained along the same lines as the proof of Theorem \ref{thm:B}.

\begin{lem}
\label{lem:exl}
Let $\f$ be a closed singular Riemannian foliations on a complete Riemannian manifold $M$, such that the space of leaves $M/\f$ is isometric to a quotient $N/\Gamma$, where $N$ is a Riemannian manifold and $\Gamma \subset \mr{Iso}(N)$ is a discrete group of isometries. Let $N_0 \subset N$ denote a fundamental domain of the $\Gamma$-action and identify $N_0 \cong M/\f$. Now assume that $S \subset N$ is a taut submanifold that is completely contained in the interior of $N_0$. Then if $Q: M \To N_0$ denotes the projection, the submanifold $Q^{-1}(S)$ is taut, too.
\end{lem}

In the case where $M = \R^{n+k}$ is the standard Euclidean space and $\f$ is an $n$-dimensional isoparametric foliation, i.e. the parallel foliation induced by an isoparametric submanifold $L$ of dimension $n$, identify a section $\Sigma$ with the Euclidean space $\R^k$. Then take a small taut submanifold $S \subset \R^k$ completely contained in the interior of a Weyl chamber associated to the finite Coxeter group generated by the reflections across the $L$-focal hyperplanes in $\Sigma$ and consider the $\f$-saturated set $T = \left\{p \in \R^{n+k} \mbox{$|$} L_p \cap S \neq \emptyset \right\}$. Then, due to the last lemma, $T$ is a taut submanifold of $\R^{n+k}$.\\

We finally come to the refined version of Theorem \ref{thm:B}.

\begin{thm}
\label{thm:infpol}
Let $\f$ be a closed singular Riemannian foliation on a complete Riemannian manifold $M$. Then $M/\f$ is an orbifold and $\f$ is taut if and only if the quotient $M/\f$ is a good Riemannian orbifold with a pointwise taut universal covering orbifold.
\end{thm}

Let us assume that $M/\f$ is a good Riemannian orbifold. Then, by Theorem \ref{thm:B} the foliation $\f$ is taut if and only if the universal covering orbifold, that is a manifold in this case, is pointwise taut. Thus we prove Theorem \ref{thm:infpol} by the observation that in the orbifold case the quotient of a taut foliation is developable.

\begin{lem}
\label{lem:firsthalf}
If $\f$ is a closed and taut singular Riemannian foliation on a complete Riemannian manifold $M$ that has an orbifold quotient, then $M/\f$ is a good Riemannian orbifold. 
\end{lem}

\begin{proof}
It is a well known fact that every Riemannian orbifold is the quotient of a regular Riemannian foliation. For instance, one could take the foliation $\widehat{\f}$ on the manifold $\widehat M$ of orthonormal frames of $M/\f$ induced by the almost free action of $O(k)$, where $k$ is the dimension of the orbifold (cf. \cite{Hae2}). By Theorem \ref{thm:B}, $\f$ is taut if and only if $\widehat{\f}$ is taut and the latter is taut if and only if its lift $\widetilde{\f}$ to the universal covering is taut as we saw above. Hence $\widetilde{\f}$ is taut, too. Since $\widetilde M$ is simply connected, $\widetilde{\f}$ has trivial holonomy (cf. Section \ref{sec:SP}) and therefore, $\widetilde M/\widetilde{\f}$ is a complete Riemannian manifold, which is also simply connected by the exact homotopy sequence. In particular, $\widetilde M/\widetilde{\f} \To M/\f$ is the universal orbifold covering.
\end{proof}

The property that the quotient of a singular Riemannian foliation is a Riemannian orbifold can be described by means of the infinitesimal foliations. Namely, this class of foliations coincides with the class of singular Riemannian foliations whose infinitesimal foliations have sections. Let us recall that a singular Riemannian foliation $(M,g,\f)$ admits \emph{sections} if there exists a complete, immersed submanifold $\Sigma_p$ through every regular point $p \in M$ that meets every leaf and always orthogonally. It is not hard to see that a section is totally geodesic in $M$. As an example, the set of orbits of a polar action is a singular Riemannian foliation admitting sections. Motivated by this example we also speak about \emph{polar foliations}. We end this section with a short recollection of the basic notions for those foliations and reformulate our result about foliations whose quotients are orbifolds.\\

Singular Riemannian foliation with sections are well understood and were studied, for example by Alexandrino and T\"oben. One nice feature of this class is that one can canonically construct a blow up which has the same horizontal geometry (cf. \cite{T}). In \cite{L} it is shown that the existence of such a \emph{geometric resolution} of a singular Riemannian foliation is equivalent to the fact that the foliation carries at the infinitesimal level the information of a singular Riemannian foliation with sections. Such foliations are called \emph{infinitesimally polar} and were first defined and discussed by Lytchak and Thorbergsson in \cite{LT}.

If $\mr{dim}(\f) =n$ and $\mr{dim}(L_p) = r$, then the infinitesimal singular foliation $\mathfrak{F}_p$ (see Section \ref{sec:SP} for the definition) on $T_pM = T_pM_{n-r} \times (T_pM_{n-r})^{\perp}$ is a product $\mathfrak{F}_p^v \times \mathfrak{F}_p^h$, where $\mathfrak{F}_p^v$ is the trivial foliation given by parallels of $T_pL_p$ and the main part $\mathfrak{F}_p^h$ on $(T_pM_{n-r})^{\perp}$ is a singular Riemannian foliation, invariant under rescalings and with the origin as the only $0$-dimensional leaf. Thus $\mathfrak{F}_p^h$ is the cone over a foliation of dimension $n-r$ on the unit sphere of $T_pM_{n-r}^{\perp}$, which is induced by the intersections of the nearby higher dimensional leaves with a slice through $p$. In particular, the foliation $\mf F_p$ is polar if and only if its factor $\mf F_p^h$ is polar.

\begin{exl}
Let $(M^m,g)$ be a complete, simply connected Riemannian manifold with a closed singular Riemannian foliation $\f$ of codimension 2 and $M/\f = S^2/\Gamma$. Then $\f$ is infinitesimally polar and $\Gamma$ is a finite Coxeter group. Further, $\f$ is taut and therefore has no exceptional leaves. Let $L \in M/\f$ be a point of codimension $2$, i.e. a corner. Take a point $p \in L^{m-k}$ and consider the infinitesimal singular Riemannian foliation $\mathfrak F_p^h$ on $(\nu_p(L),g_p)$. Since $L$ has codimension 2 in $M/\f$, the singular Riemannian foliation $\mathfrak F_p^h$ is the cone foliation over a singular Riemannian foliation of codimension 1 on the unit sphere $S^{k-1}$ in $\nu_p(L)$. By a result of M\"unzner (see \cite{Mu1}, \cite{Mu2}), one therefore has $S^{k-1}/\mathfrak F_p^h = I_d$ for an interval $I_d$ with length $|I_d| = \pi/d$ and $d \in \{1,2,3,4,6\}$, i.e. $\nu_p(L)/\mathfrak F_p^h$ is an open cone over $I_d$ with angle $\pi/d$. Note that to obtain a local isometry between $\nu_p(L)/\mathfrak F_p^h$ and $U/\f$ for a neighborhood $U$ around $p$, indeed one has to change the metric on $\nu_p(L)$, but this has no influence on the possible values of the angle, because the metrics coincide in $0$. Now the finite to one mapping $U/\f \To M/\f$ between the local and global quotient is given by the quotient $(U/\f)/W$, where $W$ is a group acting on $U$ by isometries. But the absence of exceptional leaves implies that $W$ acts trivially, so that a neighborhood of $L$ in $M/\f$ is isometric to $U/\f$. It follows by the known classification of $S^2/\Gamma$ that the quotient $M/\f$ is either the whole sphere $S^2$, the hemisphere $S^2/\Z_2$, a sickle $S^2/D_i$ with $i \in \{2,3,4,6\}$, or it is a spherical triangle with angles $(\pi/n_1,\pi/n_2,\pi/n_3)$ and $(n_1,n_2,n_3) \in \{(2,2,2),(2,2,3),(2,2,4),(2,2,6),(2,3,3),(2,3,4)\}$
\end{exl}

If $\iota: \Sigma \To M$ is a section of $\f$ then $d\iota_p(T_p\Sigma)$ is a section of $\mathfrak{F}_{\iota(p)}$. Thus a singular Riemannian foliation with sections is infinitesimally polar. Conversely, in the general case, a section $\Sigma$ of $\mathfrak{F}_p$ cannot be realized as the tangent space of a local section, because this is equivalent to the fact that the horizontal distribution over the regular stratum given by $\mathcal{H} = \bigcup_{p \in M_0} (T_pL_p)^{\perp_{g_p}}$  is integrable, what, under the assumption of completeness of $M$, is equivalent to existence of sections (cf. \cite{A}). A well known example of an infinitesimally polar singular Riemannian foliation that is not polar is given by the fibers of the Hopf fibration $S^1 \hookrightarrow S^3 \To S^2(\frac{1}{2})$.\\ 

In \cite{LT} Lytchak and Thorbergsson proved the following\\

\noindent \textbf{Theorem} (Lytchak and Thorbergsson, 2010) 
\emph{The following are equivalent:
\begin{enumerate}
\item The infinitesimal singular Riemannian foliation $\mathfrak{F}_p$ is polar;
\item $\f$ is locally closed at $p$ and a local quotient $U/\f$ of a neighborhood $U$ of $p$ is a Riemannian orbifold.\\
\end{enumerate}}

In fact, in \cite{LT} it is shown that the statements above are equivalent to the non-explosion of the curvature in the local quotients as one approaches a boundary point $p$ of $M_0$.\\

Using the description of Lytchak and Thorbergsson and Theorem \ref{thm:infpol} we obtain 

\begin{cor}
\label{cor:infpol}
Let $\f$ be a closed singular Riemannian foliation on a complete Riemannian manifold $M$. Then $\f$ is infinitesimally polar and taut if and only if the quotient $M/\f$ is a good Riemannian orbifold with a pointwise taut universal covering orbifold.
\end{cor}

If the infinitesimal singular Riemannian foliation $\mathfrak{F}_p$ on $(T_pM,g_p)$ is polar, $\mathfrak{F}_p$ is an isoparametric foliation given by the parallel foliation induced by a regular and hence isoparametric leaf. Since isoparametric foliations on a Euclidean space are taut (see \cite{HPT}), the infinitesimal foliation $\mathfrak{F}_p$ is taut in this case. The next lemma states that this is always true.

\begin{lem}
\label{lem:inftaut}
Let $\f$ be a closed singular Riemannian foliation on a complete Riemannian manifold $(M,g)$. If $\f$ is taut, then for every $p \in M$, the infinitesimal foliation $\mf F_p$ on $(T_pM,g_p)$ is taut.
\end{lem}

\begin{proof}
We will use the notation from the beginning of this section. Take a point $p \in M$ with $d = \mathrm{dim}(L_p)$ and $k = \mr{codim}(M_{n-d})$. Then we have an orthogonal splitting $T_pM = T_pM_{n-d} \oplus \nu_p(M_{n-d})$ that induces a splitting of $\mathfrak{F}_p$ into a product foliation $\mathfrak{F}_p = \mf F_p^v \times \mathfrak{F}_p^h$, where the first factor is the foliation given by parallels of $T_pL_p$ and the main part $\mf F_p^h$ is the cone-foliation, i.e. invariant under all homotheties $r_{\lambda}(v) = \lambda\cdot v$, of a singular Riemannian foliation with compact leaves of dimension at least one on the unit sphere in $\nu_p(M_{n-d})$ (if $\nu_p(M_{n-d}) \neq \left\{0\right\}$).
Since the path spaces of these product submanifolds are the products of the corresponding path spaces in the factors and the critical points are exactly the tupels of critical points in these factors, so that the critical data behave additive, we conclude that 
$\mathfrak{F}_p$ is taut if and only if $\mathfrak{F}_p^h$ is taut, because $T_pL_p$ is contractible and $\mf F_p^v$ is the trivial foliation by parallels of $T_pL_p$. Further, because $\mf F_p$ as well as the set of straight lines in $T_pM$ is invariant under homotheties it follows that $\mf F_p$ is taut if and only if the restricted foliation $\mf F_p^h|_D$ is taut, where $D$ is a small ball in $\nu_p(M_{n-d})$ around the origin. Note that such a ball $D$ is always saturated. Let $U$ be a distinguished tubular neighborhood around $p$ and set $V = \phi(U)$ and $h = \phi_*\ g$, where $\phi : U \To T_pM$ with $\phi_*(\f|_U) = \mf F_p|_V$ is an embedding as in the definition of the infinitesimal foliation at $p$. Now with respect to the metric $h$ and for a small ball $D$ around the origin in $\nu_p(M_{n-d})$, the closed singular Riemannian foliation $\mf F_p|_{(T_pM_{n-d} \times D) \cap \phi(U')}$ is taut, i.e. the saturation of $\f|_{\phi^{-1}(D)}$ in $U'$ is taut, where $U' \subset U$ is a smaller distinguished tubular neighborhood at $p$ that contains $\phi^{-1}(D)$. To see this, we can choose $U'$ so small that we have to consider only critical points $\gamma$ in $U'$ with energy $r$ such that the whole ball of radius $r$ around $\gamma(1)$ is contained in $U$, so that we have 
\begin{equation*}
P(U,L_{\gamma(0)} \cap U \times \gamma(1))^r = P(M,L_{\gamma(0)} \times \gamma(1))^r
\end{equation*}
\noindent  and we conclude by tautness of $\f$ that all the local unstable manifolds can be completed in $U$ below the energy $r$. Thus, since $U$ and $U'$ can be assumed to be diffeomorphic, the local unstable manifolds can also be completed in $U'$, what implies our claim. If we now consider the blow up metrics $h^{\lambda}$ on $V^{\lambda} = \left\{\lambda v \mbox{$|$} v \in V\right\}$ defined by
\begin{equation*}
h^{\lambda} = \lambda^2\cdot (r_{\lambda})_*\ h,
\end{equation*}
\noindent it follows that our restricted foliation is also taut with respect to the metrics $h^{\lambda}$. But the constant metric $g_p$ is just the flat limit $\lim_{\lambda \to \infty}h^{\lambda}$ and we deduce that $\mf F_p$ is taut with respect to $g_p$, because it is not hard to see that if a sequence of perfect Morse functions converge to a Morse function, this limit has to be perfect. This together with our genericity results from Section \ref{sec:SP} finish the proof.        

\end{proof}

In the standard picture of an isometric action of a Lie group $G$ on a Riemannian manifold $M$, one could ask if there exists another Riemannian manifold $(\widehat{M},\hat{g})$, canonically related to $M$, on which $G$ acts by isometries in such a way that all orbits of this action have the same dimension, because the singular leaves are the main source of difficulties. If one additionally tries to resolve the action in a way that preserves the horizontal geometry, there was no such general construction known before the work of Lytchak \cite{L}, who came up with a canonical resolution preserving the transverse geometry.

As the main result in \cite{L}, Lytchak gave an equivalent characterization of the infinitesimally polar foliations as exactly those foliations that admit a \emph{geometric resolution}, where he defined a geometric resolution of $(M,g,\f)$ to be a smooth and surjective map $F: \widehat{M} \To M$ from a Riemannian manifold $(\widehat{M},\hat{g})$ with a regular foliation $\widehat{\f}$ such that the following holds true: For all smooth curves $c$ in $\widehat{M}$ the transverse lengths of $c$ with respect to $\widehat{\f}$ and of $F(c)$ with respect to $\f$ coincide.
The transverse length of a smooth curve $c : [a,b] \To M$ is defined as the length of the projection to local quotients
\begin{equation*}
L_T(c) = \int_a^b \|P_{c(t)}(\dot{c}(t))\| dt,
\end{equation*}
where $P_q : T_qM \To (T_qL_q)^{\perp}$ denotes the orthogonal projection. In particular, $F$, as in the definition of a geometric resolution, sends leaves of $\widehat{\f}$ to leaves of $\f$, such a map is called \emph{foliated}, and induces a length preserving map between the quotients.\\

In \cite{L} Lytchak proved\\ 

\noindent \textbf{Theorem} (Lytchak, 2010) 
\emph{A singular Riemannian foliation $(M,\f)$ has a geometric resolution if and only if $\f$ is infintesimally polar. If $\f$ is infinitesimally polar, then there is a canonical resolution $F: \widehat{M} \To M$ satisfying $\mr{dim}(\widehat M) = \mr{dim}(M)$ and such that $F$ induces a diffeomorphism between the regular part $M_0$ and $F^{-1}(M_0)$ as well as an isometry between the quotients $M/\f$ and $\widehat M/\widehat{\f}$.
If $\f$ is given by the orbits of a group $G \subset \mr{Iso}(M)$, then $G$ acts by isometries on $\widehat{M}$ and $\widehat{\f}$ is the associated orbit foliation. If $M$ is compact or complete, then $\widehat M$  has the respective property .}\\

Note that horizontal in $(M,\f)$ and $(\widehat M,\widehat{\f})$ with the same projection to the common quotient have the same horizontal index (cf. Section \ref{sec:HGI}). \\ 

Let us now say some words about his canonical resolution. Recall that the \emph{Grassmannian bundle} $\mathfrak{G}_k(M)$ of a given manifold $M^{n+k}$ consists fiberwise of the Grassmaniann manifolds 
\begin{equation*}
G_k(T_pM) = \left\{ \sigma \subset T_pM \mbox{$|$} \sigma \text{ is a $k$-plane} \right\}
\end{equation*}
 of the $k$-dimensional linear subspaces of the tangent space $T_pM$, that is to say $\mathfrak{G}_k(M) = \bigcup_{p \in M} G_k(T_pM)$. For a detailed discussion of the Grassmanian bundle with its natural metric we refer the reader to \cite{W}.
\newline 
 
For a singular Riemannian foliation $\f$ of codimension $k$ on $M$ which has sections, Boulam defined in \cite{Bou} the set \begin{equation*}
\widehat{M}' = \left\{ T_p\Sigma \mbox{$|$}  \text{$\Sigma$ is a section through $p$} \right\}
\end{equation*}
 of the Grassmannian bundle. Let $P : \widehat{M}' \To M$ denote the restriction of the canonical map $\mathfrak{G}_k(M) \To M$. Boualem constructed a differentiable structure on $\widehat{M}'$ and showed that there is some Riemannian metric on $N$ such that the lifted partition $\widehat{\f}' = \left \{ P^{-1}(L) \mbox{$|$} L \in \f \right \}$ becomes a regular Riemannian foliation on $\widehat{M}'$. The foliation $\widehat{\f}'$ is called the \emph{blow up of} $\f$.
 
In \cite{T} T\"oben proved this result again with another technique and gives the following amplification:
If we denote by $h$ the natural Riemannian metric on $\mathfrak{G}_k(M)$ and by $\hat{g}' = \iota^*h$ the pull back on $\widehat{M}'$, then the pair $(\widehat{\f}',\widehat{\f}'^{\perp})$ is a bi-foliation on $\widehat{M}'$ with a Riemannian foliation $\widehat{\f}'$ and totally geodesic foliation $\widehat{\f}'^{\perp}$.
\newline

Therefore, in some sense, the sections of $\f$ play the role of a global benchmark and give rise to a resolution of the singularities. In \cite{L} Lytchak generalized this construction replacing $\widehat{M}'$ by 
\begin{equation*}
\widehat{M} = \left\{ \Sigma \subset T_pM \mbox{$|$} \text{$\Sigma$ is a section of $\mathfrak{F}_p$ through $0$} \right\}
\end{equation*}
 to infinitesimally polar singular Riemannian foliations. There is a unique Riemannian metric $\widehat g$ on $\widehat M$ such that its restriction to $T\widehat{\f} = T((P|_{\widehat M})^*\f)$ coincides with the restriction of the canonical metric on the Grassmannian bundle and $\nu_{\Sigma}(\widehat L_{\Sigma})$ is isometrically identified with $\Sigma$, where $\Sigma$ is a section of $\mf F_{P(\Sigma)}$. If we denote the restriction $P|_{\widehat M}$ by $F$, then $F: (\widehat M,\widehat F,\widehat g) \To (M,\f,g)$ is the canonical geometric resolution. Thus what is really needed for such a resolution is just the infinitesimal geometric information of a singular Riemannian foliation with sections, but not the actual existence of sections.\\
 
Using Lythak's characterization of infinitesimally polar foliations and Theorem \ref{thm:B} we therefore obtain

\begin{cor}
\label{cor:resolution}
Let $\f$ be a closed and infinitesimally polar singular Riemannian foliation on a complete Riemannian manifold $M$ and let $(\widehat M, \widehat{\f})$ denote the canonical geometric resolution. Then $\widehat{\f}$ is taut if and only if $\f$ is taut.
\end{cor}


\begin{thebibliography}{xxxxxxxx}
\bibitem[A06]{A} M. Alexandrino, \emph{Proofs of conjectures about singular Riemannian foliations}, Geom.
Dedicata \textbf{119} (2006), 219Ð234.
\bibitem[AKLM07]{DAM} M. Losik, D. Alekseevsky, A. Kriegl and P. Michor, \emph{Reflection groups on Riemannian manifolds}, Ann. Math. Pura Appl. (4) \textbf{186} (2007),  no. 1, 25-58.  
\bibitem[BG07]{BG} L. Biliotti and C. Gorodski, \emph{Polar actions on compact rank one symmetric spaces are taut},  Math. Z.  \textbf{255} (2007), no. 2, 335-342. 
\bibitem[BS58]{BS} R. Bott and H. Samelson, \emph{Application of Morse theory to symmetric spaces}, Amer. J. Math. \textbf{80}  (1958), 964-1029.
\bibitem[Bou95]{Bou} H. Boualem, \emph{Feuilletages riemanniens singuliers transversalement integrables}, Compositio Math. \textbf{95} (1995), 101-125.
\bibitem[B67]{B1} G. Bredon, \emph{Sheaf Theory}, McGraw-Hill Book Co., New York-Toronto, Ont.-London, 1967.  
\bibitem[Co71]{Co} L. Conlon, \emph{Variational completeness and K-transversal domains}, J. Differential Geometry \textbf{5} (1971), 135-147.
\bibitem[CW72]{CW} S. Carter and A. West, \emph{Tight and taut immersions}, Proc. London Math. Soc. (3) \textbf{25} (1972), 701–720. 
\bibitem[GT03]{GT} C. Gorodski and G. Thorbergsson, \emph{The classification of taut irreducible representations}, J. Reine Angew. Math. \textbf{555} (2003), 187-235.  
\bibitem[GOT04]{GOT} C. Gorodski, C. Olmos and R. Tojeiro, \emph{Copolarity of isometric actions}, Trans. Amer. Math. Soc. \textbf{356} (2004), no. 4, 1585-1608.
\bibitem[GH91]{GH} K. Grove and S. Halperin, \emph{Elliptic isometries, condition (C) and proper maps}, Arch. Math. (Basel) \textbf{56} (1991), no. 3, 288-299.
\bibitem[Hae84]{Hae2} A. Haefliger, \emph{Groupoides d'holonomie et classifiants}, Ast\'{e}risque, \textbf{116}(1984), 70-97.
\bibitem[Hae88]{Hae} A. Haefliger, \emph{Leaf closures in Riemannian foliations}, A f\^{e}te of topology, 3–32, Academic Press, Boston, MA, 1988.
\bibitem[Heb81]{He} J.J. Hebda, \emph{The regular focal locus}, J. Differential Geom. \textbf{16} (1981), no. 3, 421-429.
\bibitem[HLO06]{HLO} E. Heintze, X. Liu, and C. Olmos, \emph{Isoparametric submanifolds and a Chevalley-type restriction theorem}, Integrable systems, geometry, and topology, 151–190, AMS/IP Stud. Adv. Math., 36, Amer. Math. Soc., Providence, RI, 2006.  
\bibitem[He60]{H} R. Hermann, \emph{A sufficient condition that a map of Riemannian manifolds be a fiber bundle}, Proc. Amer. Math. Soc. \textbf{11} (1960), 236-242.
\bibitem[HPT88]{HPT} W.-Y. Hsiang, R.S. Palais, and C.-L. Terng, \emph{The topology of isoparametric submanifolds}, J. Differential Geom. \textbf{29} (1988), no. 3, 423-460.
\bibitem[IT01]{IT} J. Itoh and M. Tanaka, \emph{The Lipschitz continuity of the distance function to the cut locus}, Trans. Amer. Math. Soc. \textbf{353} (2001), no. 1, 21-40.
\bibitem[Kl82]{K} W. Klingenberg, \emph{Riemannian Geometry}, de Gruyter Studies in Mathematics, 1., Walter de Gruyter and Co., Berlin-New York, 1982. 
\bibitem[Le06]{Lei} M. Leitschkis, \emph{Pointwise taut Riemannian manifolds}, Comment. Math. Helv. \textbf{81} (2006), no. 3, 523-541.
\bibitem[L09]{L07} A. Lytchak, \emph{Notes on the Jacobi equation}, Differential Geom. Appl. \textbf{27} (2009), no. 2, 329-334. 
\bibitem[L10]{L} A. Lytchak, \emph{Geometric resolution of singular Riemannian foliations}, Geom. Dedicata \textbf{149} (2010), 379-395.
\bibitem[LT10]{LT} A. Lytchak and G. Thorbergsson, \emph{Curvature explosion in quotients and applications}, J. Differential Geom. \textbf{85} (2010), no. 1, 117-139.
\bibitem[Ma08]{Ma} F. Magata, \emph{Reductions, resolutions and the copolarity of isometric groups actions.}
M\"unstersches Informations- und Archivsystem f\"ur multimediale Inhalte, 2008. Dissertation.
\bibitem[Mi63]{M} J.W. Milnor, \emph{Morse Theory}. Annals Math. Studies, vol. 51, Princton Univ. Press, New Jersey, 1963. 
\bibitem[MM03]{MM} I. Moerdijk and J. Mrcun, \emph{Introduction to Foliations and Lie Groupoids}, Cambridge Studies in Advanced Mathematics, 91. Cambridge University Press, Cambridge, 2003. 
\bibitem[Mol88]{Mol} P. Molino, \emph{Riemannian Foliations}, Progress in Mathematics, 73. Birkh\"auser Boston, Inc., Boston, MA, 1988. 
\bibitem[Mo34]{Mo} M. Morse, \emph{The Calculus of Variations in the Large}, American Mathematical
Society Publications, vol. 18, 1934.
\bibitem[M\"{u}80]{Mu1} H. F. M\"{u}nzner, \emph{Isoparametrische Hyperfl\"{a}chen in Sph\"{a}ren}, Math. Ann. \textbf{251}  (1980), no. 1, 57-71.
\bibitem[M\"{u}81]{Mu2} H. F. M\"{u}nzner, \emph{Isoparametrische Hyperfl\"{a}chen in Sph\"{a}ren II: \"{U}ber
die Zerlegung der Sph\"{a}re in Ballb\"{u}ndel}, Math. Ann. \textbf{256} (1981), no. 2, 215-232.
\bibitem[No08]{EN} E. Nowak, \emph{Singular Riemannian Foliations: Exceptional Leaves; Tautness}, preprint (2008), math.DG/08123316.
\bibitem[Oz86]{O} T. Ozawa, \emph{On critical sets of distance functions to a taut submanifold}, Math. Ann. \textbf{276} (1986), no. 1, 91-96. 
\bibitem[PT88]{PT} R.S. Palais and C.-L. Terng, \emph{Critical Point Theory and Submanifold Geometry}, Lecture Notes in Mathematics, 1353. Springer-Verlag, Berlin, 1988. 
\bibitem[Sak96]{S}  T. Sakai, \emph{Riemannian Geometry}, Translation of Mathematical Monographs, 149. American Mathematical Society, Providence, RI, 1996.  
\bibitem[Sal88]{Sal} E. Salem, \emph{Riemannian foliations and pseudogroups of isometries}, In Appendix D of \emph{Riemannian foliations}, Birkh\"{a}user Boston, Inc.,Boston MA, 1988, 265-196.
\bibitem[Sp66]{Sp} E.H. Spanier, \emph{Algebraic topology}, McGraw-Hill Book Co., New York-Toronto, Ont.-London, 1966.
\bibitem[TT97]{TT} C.-L. Terng and G. Thorbergsson, \emph{Taut Immersions into Complete Riemannian Manifolds}, In \emph{Tight and Taut Submanifolds}, Editors T. Cecil and S.-S. Chern, Math. Sci. Res. Inst. Publ., Vol. 32, Cambridge Univ. Press, Cambridge, 1997, 181-228.
\bibitem[T88]{Tor} G. Thorbergsson, \emph{Homogeneous spaces without taut embeddings}, Duke Math. J. \textbf{57} (1988), 347-355.
\bibitem[T06]{T} D. T\"oben, \emph{Parallel focal structure and singular Riemannian foliations}, Trans. Amer. Math. Soc. \textbf{358} (2006), no. 4, 1677-1704.
\bibitem[Wa65]{Wa} F. Warner, \emph{The conjugate locus of a Riemannian manifold}, Amer. J. Math. \textbf{87} (1965), 575–604. 
\bibitem[Wa67]{Wa1} F. Warner, \emph{Conjugate Loci of Constant Order}, Ann. of Math. (2) \textbf{86} (1967), 192-212.  
\bibitem[Wa83]{Wa2} F. Warner, \emph{Foundations of Differentiable Manifolds and Lie Groups}, Graduate Texts in Mathematics, 94. Springer-Verlag, New York-Berlin, 1983.  
\bibitem[Wie08]{W} S. Wiesendorf. \emph{Aufl\"{o}sung singul\"{a}rer Bahnen isometrischer Wirkungen}. Diploma Thesis, Universit\"{a}t zu K\"{o}ln, 2008.
\bibitem[Wil07]{Wi} B. Wilking, \emph{A duality theorem for Riemannian foliations in non-negative sectional curvature}, Geom. Funct. Anal. \textbf{17} (2007), no. 4, 1297-1320.   
\end{thebibliography}
\end{document}